\documentclass[leqno]{amsart}
\usepackage{amssymb}
\usepackage{mathrsfs}
\usepackage{amsmath, amsfonts, vmargin, enumerate}
\usepackage{graphics}
\usepackage{verbatim}
\usepackage{amsthm}
\usepackage{latexsym,bm}
\usepackage{euscript}
\usepackage{dsfont}
\usepackage{color}

\input xypic
\xyoption {all}

 \makeatletter

\newtheorem{thm}{Theorem}[section]
\newtheorem{prop}[thm]{Proposition}
\newtheorem{cor}[thm]{Corollary}
\newtheorem{lem}[thm]{Lemma}
\newtheorem{defi}[thm]{Definition}
\newtheorem{remark}[thm]{Remark}
\newtheorem{example}[thm]{Example}
\newtheorem{pb}[thm]{Problem}

\newenvironment{rk}{\begin{remark}\rm}{\end{remark}}
\newenvironment{definition}{\begin{defi}\rm}{\end{defi}}

\numberwithin{equation}{section}

\newcommand{\I}{{\mathbb I}}
\newcommand{\M}{{\mathcal M}}

\renewcommand{\a}{\alpha}
\renewcommand{\b}{\beta}

\renewcommand{\d}{\delta}
\newcommand{\e}{\varepsilon}

\newcommand{\8}{\infty}
\newcommand{\el}{\ell}

\newcommand{\la}{\langle}
\newcommand{\ra}{\rangle}

\newcommand{\ep}{\varepsilon}

\newcommand{\be}{\begin{eqnarray*}}
\newcommand{\ee}{\end{eqnarray*}}
\newcommand{\beq}{\begin{equation}}
\newcommand{\eeq}{\end{equation}}
\newcommand{\beqn}{\begin{equation*}}
\newcommand{\eeqn}{\end{equation*}}

\begin{document}

\title{Noncommutative multi-parameter Wiener-Wintner type ergodic theorem}


\author{Guixiang Hong}

\author{Mu Sun}

\maketitle

\begin{abstract}
In this paper, we establish a multi-parameter version of Bellow and Losert's Wiener-Wintner type ergodic theorem for dynamical systems not necessarily being commutative. More precisely, we introduce a weight class $\mathcal{D}$, which is shown to strictly include the multi-parameter bounded Besicovitch weight class, thus including the set $$\Lambda_d=\left\{\{\lambda^{k_1}_1\dotsm\lambda^{k_d}_d\}_{(k_1,\dots,k_d)\in \mathbb{N}^d}:\quad (\lambda_1,\dots,\lambda_d) \in \mathbb{T}^d\right\};$$
then prove a multi-parameter Bellow and Losert's Wiener-Wintner type ergodic theorem for the class $\mathcal{D}$ and for noncommutative trace preserving dynamical system $(\mathcal{M},\tau,\mathbf{T})$. Restricted to consider the set $\Lambda_d$, we also prove a noncommutative multi-parameter analogue of Bourgain's uniform Wiener-Wintner ergodic theorem.

The noncommutativity and the multi-parameter induce some difficulties in the proof. For instance, our arguments in proving the uniform convergence for a dense subset turn out to be quite different since the ``pointwise'' argument does not work in the noncommutative setting; to obtain the uniform convergence in the largest spaces, we show maximal inequality between the Orlicz spaces, which can not be deduced easily using classical extrapolation argument. Junge and Xu's noncommutative maximal inequalities with optimal order, together with the atomic decomposition of Orlicz spaces, play an essential role in overcoming the second difficulty.

\end{abstract}

\section{introduction}\label{intro}
In classical ergodic theory, $(X, \mathcal{F}, \mu,T)$ is called a finite measure-preserving dynamical system if $(X, \mathcal{F}, \mu)$ is a finite measure space and $T$ a measure preserving transformation on $X$. In 1941, Wiener and Wintner \cite{WW1941} showed that for any such dynamical system $(X,\mathcal{F},\mu,T)$ and any $f \in L_1(\mu)$, there exists a set $X_f$ of full measure in $X$ such that for any $x\in X_f$, the sequence
$$\frac{1}{n+1}\sum_{k=0}^{n} \lambda^k f\left(T^k x\right)$$
converges for all $\lambda \in \mathbb{T}$ (1-dimensional complex torus).

This result proved by Wiener and Wintner is highly non-trivial. It is because the intersection of the sets  of convergence $X_{f,\lambda}$ (from the application of Dunford-Schwartz ergodic theorem associated with the contraction $\lambda T$ to the function $f$) written as $\cap_{\lambda\in\mathbb{T}}X_{f,\lambda}$ may be empty.

To describe further development of Wiener-Wintner's theorem, let us introduce the following notion. Let $\mathcal{B}(\mu)$ be a set of functions constructed from finite measure space $(X,\mathcal{F},\mu)$.

\begin{definition}\label{def:classicalWW}
A set $\mathcal{A}$ of sequences of complex numbers $a=\{a(k)\}_{k=0}^{\8}$  is called  $\mathcal{B}$-Wiener-Wintner type (in short $\mathcal{B}$-WW type), if for any measure preserving system $(X,\mathcal{F},\mu,T)$ and any $f\in \mathcal{B}(\mu)$, there exists $X_f$ of full measure in $X$, such that for any $x\in X_f$, the sequence
$$\frac{1}{n+1}\sum_{k=0}^{n} a(k) f\left(T^k x\right)$$
converges for all $a \in \mathcal{A}$.
\end{definition}

Using this definition, Wiener-Wintner's theorem can be reformulated as: {\it The set $\Lambda_1=\{(\lambda^k)_{k=0}^{\8}:\lambda\in\mathbb{T}\}$ is of $L_1$-WW type.}  To provide a full description or characterization of the largest $L_1$-WW type set (or more generally $\mathcal B$-WW type set) becomes a natural and interesting problem. Two important advances have been made since the appearance of Wiener-Wintner's theorem but there is still no description of the largest $\mathcal{B}$-WW type set for any $\mathcal{B}$ so far as we know until the moment of writing the paper.

The first important advance was made by Bellow and Losert in 1985 \cite{BL1985}, who used the notion of correlation and spectral measures of weight sequences to construct a set $\mathcal{D}$ which strictly includes bounded Besicovitch class (c.f. \cite{Bes1954}) thus the set $\Lambda_1$,  and proved in their Theorem 3.15 that $\mathcal{D}$ is of $L_1$-WW type.

The second important advance lies in Bourgain's return times theorem (c.f. \cite{Bou1} \cite{Bou2} \cite{Bou3} \cite{BFKO1989}). This was motivated by the close relationship between Wiener-Wintner type set and good universal weights. It is defined by Bellow and Losert \cite{BL1985} that a sequence of complex numbers $a=\{a(k)\}_{k=0}^{\8}$ is called a good universal weight if for any finite measure-preserving dynamical system $(Y, \mathcal{G}, \nu, S)$ and any $ g\in L_1(\nu)$, the weighted ergodic averages $$ \frac{1}{n+1}\sum_{k=0}^{n} a(k)g\left(S^k y \right) $$ converges a.e. $y\in Y$; then they proved that: For any finite measure-preserving dynamical system $(X, \mathcal{F}, \mu,T)$ under the condition that $T$ has countable Lebesgue spectrum, for every $f \in L_{\infty}(\mu)$, there exists a set of full measure $X_f$ in $X$ such that for every $x \in X_f$, the sequence $\mathbf{a}=\{f(T^kx)\}$ is a good universal weight. In 1989, Bourgain removed this extra condition and showed Bellow and Losert's result holds for all measure-preserving dynamical systems.

Bourgain's result impacts both ergodic theory and harmonic analysis, and now is called Bourgain's return times theorem. Based on his result, using the Banach principle and H\"older's inequality, it is easy to show that: For any $(X, \mathcal{F}, \mu,T)$, any $f \in L_{p}(\mu)$ and a.e. $x\in X$, $\{f(T^kx)\}$ is a good universal weight for any function from $L_q$ with $\frac{1}{p}+ \frac{1}{q}\le1$. In our language using Definition \ref{def:classicalWW}, this result can be restated as follows: Suppose $1\le p, q\le \8$ with $\frac{1}{p}+ \frac{1}{q}\le1,$  then the set
$$\bigg\{ \left\{g(S^k y)\right\}_{k=0}^{\8}:\quad \forall (Y, \mathcal{G}, \nu, S),\quad\forall g \in L_q(\nu),\quad \text{a.e. }y\in Y \bigg\}$$ is of $L_p$-WW type.  Then it is easy to observe that if we restrict the dynamical system $(Y,\mathcal{G}, \nu, S)$ to be the rotation on the torus and take $p=1$, we get Wiener-Wintner's theorem.

There are still many other important and interesting developments related to Wiener-Wintner's theorem and Bourgain's return times theorem, we refer the reader to the survey paper by Assani and Presser \cite{AP2012} for more information.

\bigskip

Motivated by the development of quantum physics, noncommutative mathematics have been taking steps forward rapidly. As an important constituent part of noncommutative analysis, noncommutative ergodic theory has been developed since the very beginning of the theory of ``rings of operators''. However at the early stage, only mean ergodic theorems have been obtained. It is until 1976 after Lance's pioneer work \cite {la-erg} that the study of individual ergodic theorems really took off. Lance proved
that the ergodic averages associated with an automorphism of a
$\sigma$-finite von Neumann algebra which leaves invariant a
normal faithful state converge almost uniformly. On the other hand, Yeadon
\cite{Ye1977} obtained a maximal ergodic theorem in the preduals of
semifinite von Neumann algebras. Yeadon's theorem provides a
maximal ergodic inequality which might be understood as a weak
type $(1,1)$ inequality. This inequality is the ergodic analogue
of Cuculescu's \cite{cu} result obtained previously for
noncommutative martingales. In contrast
with the classical theory, the noncommutative nature of these
weak type $(1, 1)$ inequalities seems a priori unsuitable for
classical interpolation arguments. The breakthrough was made in 2007 by Junge and Xu \cite{JX2007}. In this paper, they established a sophisticated real interpolation method, which together with Yeadon's weak type $(1,1)$ inequality allows them to obtain the noncommutative Dunford-Schwartz maximal ergodic theorem, thus the noncommutative individual ergodic theorem.

In the mean time, some results on noncommutative weighted ergodic theory have also appeared. By means of the Banach principle \cite{GL2000}, pointwise convergence of the bounded Besicovitch weighted averages in noncommutative $L_1$-spaces was obtained in \cite{CLS2005}. While in the multi-parameter case, some results has been shown in \cite{Ska2005} and \cite{MMT2008}. However, it is until very recently that Litvinov \cite{L2014} proved a noncommutative version of Wiener-Wintner's theorem.

Our first goal in this paper is to establish a noncommutative multi-parameter analogue of the advance made by Bellow and Losert \cite{BL1985}, as an intermediate step to fully understand noncommutative version of Bourgain's return times theorem and related results.

To state our main results, we introduce the system of notation on multi-parameter noncommutative dynamical system. Let $\mathbb{C}$ be the set of all complex numbers, $\mathbb{R}$ all real numbers, $\mathbb{N}$ all natural numbers (including $0$), $d$ be any positive integer and $\mathbb{T}^d$ the complex torus of dimension $d$:\be
\mathbb{T}^d = \{ \mathbf{z}=(z_1,\dots,z_d) \in \mathbb{C}^d : ~ |z_j|=1, \ j=1,\dots,d \}.
\ee

Let $\mathcal{M}$ be a von Neumann algebra equipped with a normal faithful finite
trace $\tau$. In this case, we will often assume that $\tau$ is normalized, i.e., $\tau(1)=1,$ thus we call $(\M,\tau)$ a noncommutative probability space. Let $L_p(\M)(1\le p\le\8)$ be the associated noncommutative $L_p$-space and $L_p\log^r L(\M)(r>0)$ the associated noncommutative Orlicz space. The set of all projections in $\M$ is denoted as $P(\M).$

Let $\mathbf{T}=(T_1,\cdots,T_d)$ be a vector of $d$ mutually commuting normal trace preserving $\ast$-automorphisms on $\mathcal{M}$. Every $T_j$ naturally extends to a trace preserving $\ast$-automorphism on $L_p(\M)$ for all $1\le p< \8$ (c.f. Lemma 1.1 \cite{JX2007}). Then we call $(\M,\tau,\mathbf{T})$ a finite trace preserving dynamical system. The dynamical system is called ergodic if at least one $T_j$ ($1 \le j\le d$) is ergodic on $L_2(\mathcal{M})$, that is, for any $x \in L_2(\mathcal{M})$, $Tx=x$ implies that $x= c \cdot \mathbb{I}, ~c\in \mathbb{C}$ where $\mathbb{I}$ is the identity in $\M$.

For any $\mathbf{k}=(k_1,\dots,k_d)\in \mathbb{Z}^d$, any $\mathbf{z}=(z_1,\dots,z_d)\in \mathbb{C}^d$ and any $\mathbf{T}=(T_1,\cdots,T_d)$, $\mathbf{z}^\mathbf{k}$, $\mathbf{T}^\mathbf{k}$ and $|\mathbf{k}|$ are defined as follows: $\mathbf{z}^\mathbf{k}=z_1^{k_1}\cdots z_d^{k_d}$, $\mathbf{T}^\mathbf{k}=T_1^{k_1}\cdots T_d^{k_d}$, $|\mathbf{k}|=|k_1|\cdots |k_d|.$

The ergodic averages of our multiple map $\mathbf{T}$ are defined as follows:
\be
M_{\mathbf{n}}(\mathbf{T}) = \frac{1}{|\mathbf{n+1}|} \sum_{\mathbf{k}=\mathbf{0}}^{\mathbf{n}} \mathbf{T}^{\mathbf{k}}  =   \frac{1}{(n_1+1) \cdots (n_d+1)} \sum_{k_1=0}^{n_1}\cdots \sum_{k_d=0}^{n_d} T_1^{k_1}\cdots T_d^{k_d}, \ ~\ \mathbf{k,n}\in \mathbb{N}^d.
\ee
To simplify the notation, whenever no confusion arises we will write
\beq
M_{\mathbf{n}}(x)= M_{\mathbf{n}}(\mathbf{T})(x).
\eeq
for any $x\in L_p(\M)~(1\le p\le \8)$.
For any multi-parameter weight sequence $\mathbf{a}= \{a(\mathbf{k})\in \mathbb{C}\}_{\mathbf{k}\in \mathbb{N}^d}$, the corresponding weighted ergodic averages of $\mathbf{T}$ acting on some $x\in L_p(\M)~(1\le p\le \8)$ is defined as:
\beq
M_{\mathbf{n}}(x, \mathbf{a}) = \frac{1}{|\mathbf{n+1}|} \sum_{\mathbf{k}=\mathbf{0}}^{\mathbf{n}} a(\mathbf{k}) \mathbf{T}^{\mathbf{k}} x, \ ~\ \mathbf{n}\in \mathbb{N}^d.
\eeq
In particular, if $a(\mathbf{k})=\underline{\lambda}^{\mathbf{k}}= \lambda_1^{k_1}\cdots \lambda_d^{k_d}$, $\underline{\lambda}\in \mathbb{T}^d$, then we write
\beq
M_{\mathbf{n}}(x, \underline{\lambda}) = \frac{1}{|\mathbf{n+1}|} \sum_{\mathbf{k}=\mathbf{0}}^{\mathbf{n}} \underline{\lambda}^{\mathbf{k}} \mathbf{T}^{\mathbf{k}} x, \ ~\ \mathbf{n}\in \mathbb{N}^d.
\eeq

We also define the convergence of a multi-parameter sequence $\{x_\mathbf{n}\}_{\mathbf{n}\in \mathbb{N}^d}$ in any Banach space. We say $x_{\mathbf{n}}$ converges to ${x}$ and denote as $\lim_{\mathbf{n}}x_\mathbf{n} = {x}$, if given any $\e>0$, there is an $\mathbf{N}\in \mathbb{N}^d$ such that for all $\mathbf{n}>\mathbf{N}~(n_j>N_j, j=1,2,\cdots,d)$, $\|x_{\mathbf{n}} - {x}\|< \e$.

The notion of Lance's almost uniform convergence is a nice analogue of the usual pointwise convergence, and plays an important role in individual ergodic theorem. We recall this concept in multi-parameter case. Let $x_{\mathbf{n}},x\in L_0(\M)$, the space of measurable operators (see e.g. Section 1.4 of \cite{Xu07}).  A multi-parameter sequence $\{x_{\mathbf{n}}\}$ is said to converge bilaterally almost uniformly (resp. almost uniformly) to ${x}$, if for any $\e>0$, there exists $e\in P(\M)$, such that
$$ \tau(e^\bot)\le \e \text{ and } \{e(x_{\mathbf{n}}-x)e\} ~(\text{resp. }\{(x_{\mathbf{n}}-x)e\})$$
converges to $0$ in $\M$.
Usually we denote it as b.a.u. (resp. a.u.) convergence.

\bigskip

Let $\mathcal{B}(\mathcal M)$ be a subset of $L_0(\M)$ constructed from a given probability space $(\M,\tau)$. Motivated by the concept of a.u. convergence, we define the noncommutative Wiener-Wintner type weights as follows.

\begin{definition}
A set $\mathcal{A}$ of $d$-parameter sequences of complex numbers $a=\{a(\mathbf{k})\}_{k\in\mathbb{N}^d}$  is called  $\mathcal{B}$-noncommutative bilateral Wiener-Wintner type (resp. noncommutative Wiener-Wintner type),  in short $\mathcal{B}$-NCbWW type (resp. $\mathcal{B}$-NCWW type), if for any finite trace preserving dynamical system $(\M,\tau,\mathbf{T})$ and for any $x\in \mathcal{B}$ and any $\e > 0$, there exists $e\in P(\M)$, such that
$$ \tau(e^\bot)\le \e \text{ and } \{eM_{\mathbf{n}}(x,a)e \}\quad (\text{resp.}\quad \{M_{\mathbf{n}}(x,a)e \})$$
converges in $\M$ for all $\mathbf{a} \in \mathcal{A}$.
\end{definition}

\bigskip

The multi-parameter generalization of Bellow and Losert's class \cite{BL1985} is still denoted by $\mathcal{D}$ (see the exact definition in Section 3). Now we can formulate our main result.

\begin{thm}\label{th:Main 1}
The class $\mathcal{D}$ is of $L_1\log^{2(d-1)}L$-NCbWW type and $L_2\log^{2(d-1)}L$-NCWW type.
\end{thm}

As one expect, a natural consequence follows using some standard arguments (see the end of Section 3).

\begin{cor}\label{cor:NCWW}
Let $(\M, \tau, \mathbf{T})$ be a finite trace preserving dynamical system. Then for any $x\in L_1\log^{2(d-1)}L(\M)$ and any $\e >0$, there exists $e\in P(\M)$ such that
$\tau(e^\bot)\le \e$ and
$$\|e[M_{\mathbf{n}}(x, \mathbf{a})- F(x,\mathbf{a})]e\|_\8 \rightarrow 0$$ for all $\mathbf{a}\in \mathcal{D},$
where $F(x,\mathbf{a})\in L_1\log^{2(d-1)}L(\M)$ is the b.a.u. limit of $\{M_{\mathbf{n}}(x, \mathbf{a})\}$.

Similarly, for any $x\in L_2\log^{2(d-1)}L(\M)$ and any $\e >0$, there exists $e\in P(\M)$ such that
$\tau(e^\bot)\le \e$ and
$$\|[M_{\mathbf{n}}(x, \mathbf{a})- F(x,\mathbf{a})]e\|_\8 \rightarrow 0$$ for all $\mathbf{a}\in \mathcal{D}$, where $F(x,\mathbf{a})\in L_2\log^{2(d-1)}L(\M)$ is the a.u. limit of $\{M_{\mathbf{n}}(x, \mathbf{a})\}$.
\end{cor}

Restricted to the one-parameter case, our result is more general than the noncommutative Wiener-Wintner's theorem obtained by Litvinov \cite{L2014}, since the set $\Lambda_1$ is strictly contained in $\mathcal {D}$. In the multi-parameter case, Theorem \ref{th:Main 1} is much stronger than the weighted pointwise ergodic theorem associated to the bounded Besicovitch class established in \cite{MMT2008} from the three aspects. Firstly, we show the convergence holds for the functions in $L_1\log^{2(d-1)}L$ (resp. $L_2\log^{2(d-1)}L$), which includes all the $L_p$ spaces with $p>1$ (resp. $p>2$) and is the largest space for the result to be true in terms of Junge and Xu's optimal noncommutative maximal inequality; Secondly, the present convergence is uniform with respect to the weights; Lastly, the bounded Besicovitch class is strictly included in $\mathcal D$ as shown in Section 5.

The multi-parameter and noncommutativity present some new difficulties, even though the main sketch of the proof follows from the one provided by Bellow and Losert. Their strategy starts with the introduction of the notion of correlation in defining $\mathcal D$, which unfortunately can not be directly generalized to the multi-parameter case. In Section 2.2, we introduce a new definition which coincides with the one in the one-parameter case. Then a commutative multi-parameter version of Wiener-Wintner's theorem follows from the relation between two weight sequences, see Section 2.2. However since the notion ``point'' disappears in the general noncommutative setting, the ergodic average $M_{\mathbf{n}}(x, \mathbf{a})$ cannot be reduced to the affinity of spectral measures of two weight sequences. Thus in Section 2.3, we introduce the notion of correlation associated to one noncommutative dynamical system for elements in $L_2$, and the final arguments in Section 3 are very different. Also due to the noncommutativity, classical extrapolation argument does not work in establishing multi-parameter maximal inequality. In the present paper, we make use of Tao's atomic decomposition of Orlicz space \cite{Tao2001} to carry out the extrapolation argument. It turns out that our argument is much easier than that in \cite{Hu2009}, which seems only applicable to some special cases.

However, our result is not so satisfactory in the sense that our $\mathbf{T}$ is restricted to be automorphisms. But in the classical case, $\mathbf{T}$ could be any Dunford-Schwartz operators due to the dilation theory, see for instance \cite{BO1983}. The reader can find the classical dilation theory in \cite{Pel76}. But the noncommutative counterpart fails for completely contractive operators, see \cite{JRS10} and the references therein. Therefore it remains an open problem that whether our definition of $\mathcal{B}$-NCWW type is equivalent to the one defined through completely contractive operators.
\bigskip

Considering the multi-parameter version of the set $\Lambda_1$, we obtain the noncommutative multi-parameter Bourgain's uniform Wiener-Wintner theorem \cite{Bou3}.

\begin{thm}\label{th:Main 2}
Given an ergodic finite trace preserving system $(\M, \tau, \mathbf{T})$, define $\mathcal{K}$ as the $\|\cdot\|_2$-closure of $E$---the linear span of $ \bigcap_{j=1}^d E_j$ where
\be
E_j= \{ x \in L_2(\mathcal{M})| T_j(x)= \mu_j x \text{ for some }\mu_j \in \mathbb{T} \},j=1,\cdots,d.
\ee For any $x\in \mathcal{K}^{\bot}\cap \M$ and any $\e>0$, there exists $e\in P(\M)$ that
$$ \tau(e^\bot)\le \e \text{ and } \lim_{\mathbf{n}} \sup_{\underline{\lambda}\in \mathbb{T}^d} \left\| M_{\mathbf{n}}(x,\underline{\lambda})e \right\|_\8 = 0.
$$
\end{thm}

Theorem \ref{th:Main 2} is in a sense stronger than Thoerem \ref{th:Main 1}, since we get convergence of $\sup_{\underline{\lambda}\in\mathbb{T}^d} \left\| M_{\mathbf{n}}(x,\underline{\lambda})e \right\|_\8 $. In one-parameter case, Litvinov used a noncommutative Van der Corput's inequality, which was discovered in \cite{NSZ2005}, to obtain the convergence independent of weights in $\Lambda_1$. In the present paper, we firstly establish a multi-parameter Van der Corput's inequality, which cannot be deduced from one-parameter case by iteration due to the noncommutativity. Then the rest of argument relies on a spectral characterization of $\mathcal{K}^{\bot}$ and the ergodic theorems by Junge and Xu \cite{JX2007}.

\section{preliminaries}\label{pre}

\subsection{Noncommutative vector-valued $L_p$ spaces}
We use standard notions for the theory of noncommutative $L_p$ spaces. Our main references are \cite{PX2003} and \cite{Xu07}.
Let $\mathcal{M}$ be a von Neumann algebra equipped with a normal finite
faithful trace $\tau$. Let $L_0(\M)$ be the spaces of measurable operators associated to $(\M,\tau)$. For a measurable operator $x$, its generalized singular number is defined as
$$\mu_t(x) = \inf \{ \lambda > 0 : \tau \big( \mathds{1}_{(\lambda,\8)}(|x|)\big)\le t\}, \ \ ~ t>0.$$
The trace $\tau$ can be extended to the positive cone $L_0^+(\M)$ of $L_0(\M)$, still denoted by $\tau$,
$$\tau(x) = \int_0^\8 \mu_t(x) dt, \ \ ~ x\in L_0^+(\M).$$
Given $0 < p < \8,$ let
$$L_p(\M) = \{x \in L_0(\M) : \tau (|x|^p) < \8\}$$
and for $x \in L_p(\M),$
$$\|x\|_p = \big(\tau (|x|^p)\big)^{\frac1p} = \big(\int_0^\8 (\mu_t(x))^p dt \big)^{\frac1p}.$$
Then $(L_p(\M),\|\cdot\|_p)$ is a Banach space (or quasi-Banach space when $p<1$). This is the noncommutative $L_p$ space associated with $(\M,\tau)$, denoted by $L_p(\M,\tau)$ or simply by $L_p(\M)$.  As usual, we set $L_\infty(\M,\tau)=\M$ equipped with the operator norm.

The noncommutative Orlicz spaces are defined in a similar way as commutative ones. Given an Orlicz function $\Phi$, the Orlicz
space $L_\Phi(\M)$ is defined as the set of all measurable operators $x$ such
that $\Phi( \frac{|x|}{\lambda} ) \in L_1(\M)$ for some $ \lambda>0$. Equipped with the norm
$$ \|x\|_\Phi = \inf\left\{ \lambda>0: \tau \left[ \Phi\left( \frac{|x|}{\lambda} \right) \right] \le 1\right\},$$
$L_\Phi(M)$ is a Banach space.
When $\Phi(t)=t^p$ with $1\leq p<\infty$, the space $L_\Phi(\M)$ coincides with $L_p(\M)$. If $\Phi(t)= t^p(1+ \log^+t)^r$ for $ 1\le p <\8 $ and $r>0$, we get the space $L_p\log^r L(\M)$. From the definition, it is easy to check that whenever $(\M,\tau)$ is a probability space, we have
$$L_q(\M)\subset L_p\log^r L(\M)\subset L_s(\M)$$
for $q>p\geq s\geq1$.

\bigskip

The spaces $L_p(\M; \ell_\8)$ and $L_p(\M; \ell_\8^c)$ play an important role in the formulation of noncommutative maximal inequalities. A sequence
$ \{x_n\}_{n\ge 0} \subset L_p(\M)$ belongs to $L_p(\M; \ell_\8)$ if and only if it can be factored
as $x_n = a y_n b$ with $a, b \in L_{2p}(\M)$ and a bounded sequence $\{y_n\} \subset L_\8(\M)$. We
then define
$$\|\{x_n\}_n\|_{L_p(\ell_\8)}=\inf_{x_n = a y_n b}\big\{\|a\|_{2p}\,
 \sup_{n\ge0}\|y_n\|_\8\,\|b\|_{2p}\big\} .$$
Following \cite{JX2007}, this norm is symbolically denoted by $\|{\sup_n}^+ x_n\|_p$. It is shown in \cite{JX2007} that a positive
sequence $\{x_n\}_n$ belongs to $L_p(\M; \ell_\8)$ if and only if there exists
$a \in L^+_p (\M)$ such that $x_n \le a$ for all $n \ge 0$. Moreover

$$\|{\sup_n}^+ x_n\|_p = \inf\big\{\|a\|_{p}\;:\; a\in L^+_p(\M)
 \;\mbox{s.t.}\; x_n\le a,\ \forall\;n\ge0\big\}.$$
Here and in the rest of the paper, $L^+_p (\M)$ denotes the positive cone of $L_p(\M).$ The space
$L_p(\M; \ell_\8^c)$ is defined to be the set of sequences $\{x_n\}_{n\ge 0}$ with $\{|x_n|^2\}_{n\ge 0}$ belonging to $L_{p/2}(\M; \ell_\8)$ equipped with (quasi) norm
$$\|\{x_n\}_n\|_{L_p(\ell^c_\8)}=\|\{|x_n|^2\}_n\|^{\frac12}_{L_{\frac p2}(\ell_\8)}.$$
We refer to \cite{Jun2002}, \cite{JX2007} and \cite{Musat2003} for more information on these spaces and related facts.

The vector-valued Orlicz spaces $L_p\log^r L(\M; \ell_\8)$  $(1\le p< \8,\ r>0)$ are firstly introduced by Bekjan {\it et al} in \cite{BCO2013}. The key observation is that ${L_p(\ell_\8)}$-norm has an equivalent formulation:
$$\|\{x_n\}_n\|_{L_p(\ell_\8)}=\inf \big\{\frac12(\|a\|^2_{2p}+\|b\|^2_{2p})\,
 \sup_{n\ge0}\|y_n\|_\8\big\},$$
where the infimum is taken over the same parameter as before. Given an Orlicz function $\Phi$, let $\{x_n\}$ be a sequence of operators in $L_{\Phi}(\M)$, we define
$$\tau\big(\Phi({\sup_n}^+ x_n)\big)=\inf \big\{\frac12\big(\tau(\Phi(|a|^2))+\tau(\Phi(|b|^2))\big)\,
 \sup_{n\ge0}\|y_n\|_\8\big\},$$
where the infimum is taken over all the decompositions $x_n=ay_nb$ for $a,b\in L_0(\M)$ and $y_n\in L_\infty(\M)$ with $|a|^2,|b|^2\in L_\Phi(\M)$ and $\sup_n\|y_n\|_\infty\leq1$. Then $L_\Phi(\M;\ell_\infty)$ is defined to be the set of sequences $\{x_n\}_n\subset L_\Phi(\M)$ such that there exists one $\lambda>0$ satisfying $$\tau\big(\Phi({\sup_n}^+ \frac{x_n}\lambda)\big)<\infty$$
equipped with the norm
$$ \|\{x_n\}_n\|_{L_\Phi(\ell_\infty)} = \inf\left\{ \lambda>0: \tau\big(\Phi({\sup_n}^+ \frac{x_n}\lambda)\big)<1\right\}.$$
Then $(L_\Phi(\M;\ell_\infty),\|\cdot\|_{L_\Phi(\ell_\infty)})$ is a Banach space.
It was proved in \cite{BCO2013} that a similar characterization holds for sequences of positive operators:
$$\tau\big(\Phi({\sup_n}^+ x_n)\big) \thickapprox \inf\big\{\tau(\Phi(a))\;:\; a\in L^+_\Phi(\M)
 \;\mbox{s.t.}\; x_n\le a,\ \forall\;n\ge0\big\}$$
which implies a similar characterization for the norm
$$\|\{x_n\}_n\|_{L_\Phi(\ell_\infty)} \thickapprox \inf\big\{\|a\|_{\Phi}\;:\; a\in L^+_\Phi(\M)
 \;\mbox{s.t.}\; x_n\le a,\ \forall\;n\ge0\big\}.$$
From the definition, it is not difficult to check that whenever $(\M,\tau)$ is a probability space, we have
$$L_q(\M;\ell_\infty)\subset L_p\log^r L(\M;\ell_\infty)\subset L_s(\M;\ell_\infty)$$
for $q>p\geq s\geq1$. We refer the reader to \cite{BCO2013} for more information on the vector-valued Orlicz spaces.

\subsection{Correlation and spectral measure associated to a weight}

In the one parameter case, as it is defined in \cite{BL1985}, a weight sequence $\{a(k)\}_{k=0}^\8$ has a correlation if
$$ \gamma_a(m)=\lim_n \frac{1}{n+1} \sum_{k=0}^{n}\overline{a(k)}a(k+m)$$ exists for all $m\in \mathbb{N}.$ Then one can extend $\gamma_a$ to a positive definite function on $\mathbb{Z}$ by setting $\gamma_a(-m)=\overline{\gamma_a(m)}$, $m\in \mathbb{N}$, and the spectral measure on $\mathbb{T}$ associated with $a$ comes from Bochner's theorem. However, in the multi-parameter case, this way of extending $\gamma_\mathbf{a}$ to be defined on $\mathbb{Z}^d$ simply does not work. Instead, we start with extending the weight sequence to be defined on $\mathbb{Z}^d$, then define its correlation on the whole $\mathbb{Z}^d$ using two-sided averages which is positive definite, thus the associated spectral measure is defined naturally by Bochner's theorem. Let us give some necessary details.

Let $\mathbf{a}=\{a(\mathbf{k})\}_{\mathbf{k}\in \mathbb{N}^d}$ be a multi-parameter weight sequence. We extend $\mathbf{a}$ to a complex function on $\mathbb{Z}^d$ by setting $a(\mathbf{k})=0$ if $\mathbf{k}\in \mathbb{Z}^d \setminus \mathbb{N}^d.$ For $\mathbf{m}\in \mathbb{Z}^d$, the $\mathbf{m}$-translation of a function $\mathbf{a}$ is defined as $\mathbf{a}^{\mathbf{m}}= \{a(\mathbf{k+m})\}_{k\in \mathbb{Z}^d}$.

\begin{definition}\label{def:SpaceS}
Given a multi-parameter weight sequence $\mathbf{a}$ on $\mathbb{N}^d$, we say it has a correlation if
\be\begin{split}
\gamma_\mathbf{a}(\mathbf{m}) & = \lim_{\mathbf{n}} \frac{1}{ | \mathbf{n+1} |} \sum_{\mathbf{k} = - \mathbf{n}}^{\mathbf{n}} \overline{a( \mathbf{k} )} a ( \mathbf{k}+\mathbf{m} )\\
& = \lim_{\mathbf{n}} \frac{1}{ | \mathbf{n+1} |} \sum_{k_1 = \max\{ 0,-m_1 \}}^{n_1} \cdots \sum_{k_d = \max\{ 0,-m_d \}}^{n_d} \overline{a( \mathbf{k} )} a ( \mathbf{k}+\mathbf{m} )
\end{split}\ee
exists for all $\mathbf{m}\in \mathbb{Z}^d.$
We define the space $S$ to be the set of all sequences having correlations.
\end{definition}

\begin{rk}\label{rk:spectralmeasure}
Restricted to the one-parameter case, it is easy to check that our new definition of correlation coincides with the one given in \cite{BL1985}, which was first introduced by Wiener. It is also easy to verify that the correlation function $\gamma_{\mathbf{a}}$ is positive definite on $\mathbb{Z}^d$ (see Proposition \ref{prop:translation} in Section 5 for details). Then by Bochner's Theorem \cite{Rudin1962}, we have a representation
\be
\gamma_{\mathbf{a}}(\mathbf{m}) = \int_{\mathbb{T}^d} \mathbf{z}^{\mathbf{m}} d \sigma_{\mathbf{a}}(\mathbf{z}), \ \ \ \ \ \mathbf{m}\in \mathbb{Z}^d,
\ee
where the positive Borel measure $\sigma_{\mathbf{a}}$ is uniquely determined by $\mathbf{a}$; we call $\sigma_{\mathbf{a}}$ the {spectral measure associated with the weight} $\mathbf{a} \in S$.
\end{rk}

In the following, we introduce some multi-parameter analogues of the approximation results related to the spectral measure theory. We omit the proofs, since the arguments are parallel to the ones  in the one parameter case, see for instance \cite{CKF1977}.

\begin{definition}\label{def:n-measure}
Let $\mathbf{a}=\{a(\mathbf{k})\}_{\mathbf{k}\in \mathbb{N}^d}$. For each $\mathbf{n}\in \mathbb{N}^d $, we define
the positive Borel measure $\sigma_{\mathbf{a}}^{\mathbf{n}}$ on the torus $\mathbb{T}^d$ by
\be
\sigma_{\mathbf{a}}^{\mathbf{n}}(E)= \int_E \frac{1}{|\mathbf{n+1}|} \left| \sum_{\mathbf{k}=\mathbf{0}}^{\mathbf{n}} a(\mathbf{k})\overline{\mathbf{z}}^{\mathbf{k}} \right|^2 d\mathbf{z}
\ee
for any Borel set $E \subset \mathbb{T}^d$, where $d\mathbf{z}$ is the normalized Haar measure on $\mathbb{T}^d$.
\end{definition}

\begin{prop}\label{prop:measureconvergence}
Let $\mathbf{a} \in S$. Then the measure $\sigma_{\mathbf{a}}^{\mathbf{n}}$ converges weakly to $\sigma_{\mathbf{a}}$.
\end{prop}

\begin{definition}\label{def:affinity}
Let $\mathbf{M}^+(\mathbb{T}^d)$ be the set of all bounded positive Borel measures on $\mathbb{T}^d$. Let $P$, $Q \in \mathbf{M}^+(\mathbb{T}^d)$, let
$\nu \in \mathbf{M}^+(\mathbb{T}^d)$ such that $P$ and $Q$ are absolutely continuous with respect to $\nu$. Then {\rm the affinity} of $P$ and $Q$ is defined by
\be
\rho(P,Q) = \int_{\mathbb{T}^d}\left( \frac{dP}{d\nu}\right)^{1\slash 2} \left( \frac{dQ}{d\nu}\right)^{1\slash 2} d\nu.
\ee
\end{definition}
As in the one-parameter case, it is easy to check that  $\rho(P,Q)$ is independent of $\nu$ and $\rho(P,Q)=0$ if and only if $P$ and $Q$ are mutually singular, which is denoted by $P \bot Q$.

\begin{prop}\label{prop:approxaffinity}
Let $\{P_{\mathbf{n}}\}$, $\{Q_{\mathbf{n}}\}$ be two multi-parameter sequences of elements in $\mathbf{M}^+(\mathbb{T}^d)$, and $P$, $Q \in \mathbf{M}^+(\mathbb{T}^d)$. If the sequences  $\{P_{\mathbf{n}}\}$ and $\{Q_{\mathbf{n}}\}$ converge weakly to $P$ and $Q$ respectively, then
\be
\limsup_{\mathbf{n}} \rho(P_{\mathbf{n}},Q_{\mathbf{n}}) \le \rho(P,Q).
\ee
\end{prop}

Let the above measures be specialized to the ones arising from $\mathbf{a}$ and $\mathbf{b}$ in $S$, namely taking $P= \sigma_{\mathbf{a}}$, $Q= \sigma_{\mathbf{b}}$, $P_{\mathbf{n}}= \sigma_{\mathbf{a}}^{\mathbf{n}}$ and $Q_{\mathbf{n}}= \sigma_{\mathbf{b}}^{\mathbf{n}}$, we have the following corollary of Proposition \ref{prop:measureconvergence} and \ref{prop:approxaffinity}.

\begin{cor}\label{cor:sequenffinity}
Let $\mathbf{a}$ and $\mathbf{b} \in S$. Then
\be
\limsup_{\mathbf{n}} \frac{1}{|\mathbf{n+1}|} \left| \sum_{\mathbf{k}=\mathbf{0}}^{\mathbf{n}} a(\mathbf{k})\overline{b(\mathbf{k})} \right|
\le \rho(\sigma_{\mathbf{a}},\sigma_{\mathbf{b}}).
\ee
In particular,
\be
\lim_{\mathbf{n}} \frac{1}{|\mathbf{n+1}|} \sum_{\mathbf{k}=\mathbf{0}}^{\mathbf{n}} a(\mathbf{k})\overline{b(\mathbf{k})} = 0
\ee
if $\sigma_{\mathbf{a}}$ and $\sigma_{\mathbf{b}}$ are mutually singular.
\end{cor}

Specializing Corollary \ref{cor:sequenffinity} further---by taking $b(\mathbf{k})=\mathbf{z}^{\mathbf{k}}$ for $\mathbf{z} \in \mathbb{T}^d$, it is a direct calculation that the correlation $\gamma_{b}(\mathbf{m}) = z^{\mathbf{m}}$. So by the Fourier inverse transform we get the spectral measure $\sigma_{b}= \delta_{z}$, which is the Dirac measure at $z$. Then we have the following result.

\begin{cor}\label{cor:WW}
Let $\mathbf{a} \in S.$ Then for each $\mathbf{z} \in \mathbb{T}^d$ we have
\be
\limsup_{\mathbf{n}}\frac{1}{|\mathbf{n+1}|} \left| \sum_{\mathbf{k}=\mathbf{0}}^{\mathbf{n}} a(\mathbf{k})\overline{\mathbf{z}}^{\mathbf{k}} \right|
\le [ \sigma_{\mathbf{a}}(\{ \mathbf{z} \}) ]^{1\slash 2}.
\ee
\end{cor}

Corollary \ref{cor:WW} is quite suggestive in deducing the following multi-parameter version of classical Wiener-Wintner's theorem. A $d$-parameter  commutative dynamical system is denoted by $(X,\mathcal{F},\mu,\mathbf{T})$, where $\mathbf{T}$ is a $d$-tuple of mutually commuting measure preserving transformations on finite measure space $X$. The $d$-parameter version of $\mathcal B$-WW type is defined in a similar way as  Definition \ref{def:classicalWW}.

\begin{prop}\label{prop:WWErgodic}
$\Lambda_d = \left\{ \{ \underline{\lambda}^{\mathbf{k}}\}_{\mathbf{k}\in \mathbb{N}^d}\ :  \underline{\lambda} \in \mathbb{T}^d \right\}$ is of $L_1\log^{d-1}L$-WW type.

\end{prop}
\begin{proof}
Fix any $d$-parameter commutative dynamical system $(X,\mathcal{F},\mu,\mathbf{T})$. Using decomposition into ergodic parts, we can assume the dynamical system is ergodic. For any $\underline{\lambda}\in \mathbb{T}^d$, we have $\sigma_{\{\underline{\lambda} ^{\mathbf{k}}\}}(\mathbf{m}) = \delta_{\underline{\lambda}}$. Now let $g \in L_2(\mu)$, from the individual ergodic theorem, we know that there exists a set $X_g$ of full measure satisfying $$\lim_{\mathbf{n}}\frac{1}{\mathbf{n+1}} \sum_{\mathbf{k}=0}^{\mathbf{n}}g(\mathbf{T^{k+m}}\omega)\bar{g}(\mathbf{T^k}\omega)=\langle \mathbf{T^m}g,g \rangle$$ for all $\omega \in X_g$ and all $\mathbf{m}\in \mathbb{Z}^d.$ So for all $\omega\in X_g$, the sequence $\{g(\mathbf{T^k}\omega)\}$ has a correlation and we have the same spectral measure $\sigma_g$.
Then let $E$ be the span of eigenfunctions of $\mathbf{T}$, and let $V = \{g\in L_2(\mu): \sigma_g \text{ is continuous}\}$. Clearly $E + V$ is dense in $L_2(\mu).$ On the other hand, obviously for $g\in E$, the limit
$$\lim_{\mathbf{n}}\frac{1}{\mathbf{n+1}} \sum_{\mathbf{k}=0}^{\mathbf{n}}\underline{\lambda}^{\mathbf{k}}g(\mathbf{T^k}\omega) \;\mathrm{exists}\;\mathrm{a.e.}$$ And for $g\in V$, by Corollary \ref{cor:sequenffinity}  $$\lim_{\mathbf{n}}\frac{1}{\mathbf{n+1}} \sum_{\mathbf{k}=0}^{\mathbf{n}}\underline{\lambda}^{\mathbf{k}}g(\mathbf{T^k}\omega) = 0$$ for all $\omega \in X_g$. Now since in this classical multi-parameter case the maximal inequality and Banach principle are well known results, we deduce that $\Lambda_d$ is of $L_1\log^{d-1}L$-WW type.
\end{proof}

\subsection{Correlation and spectral measure associated to an operator in a dynamical system}
Let $(\M, \tau, \mathbf{T})$ be a finite trace preserving system. Since all $T_j (j=1,\cdots,d)$ are unitaries on $L_2(\M)$, for $x\in L_2(\M)$, we can consider the sequence of operators $\{\mathbf{T}^{\mathbf{k}}x\}_{\mathbf{k}\in \mathbb{Z}^d}$.

\begin{definition}\label{def:noncommucorrelation}
Given any $x \in L_2(\mathcal{M})$, we define the correlation as
\be
\gamma_x(\mathbf{m}) = (\mathbf{T}^{\mathbf{m}} x,x )_{\tau}=\tau((\mathbf{T}^{\mathbf{m}} x)^*x), \ \ \ \ \ \mathbf{m}\in \mathbb{Z}^d.
\ee
\end{definition}

\begin{lem}\label{lem:noncommuposidefi}
Let $x \in L_2(\mathcal{M})$, then its correlation $\gamma_x(\mathbf{m})$ is a positive definite function on $\mathbb{Z}^d$, thus there exists
a positive Borel measure $\sigma_x$ on $\mathbb{T}^d$ such that
\be
\gamma_x(\mathbf{m})= \widehat{\sigma_x}(\mathbf{m})=\int_{\mathbb{T}^d}\mathbf{z}^{\mathbf{m}}d\sigma_x(\mathbf{z}).
\ee
\end{lem}

\begin{proof}
Let $\forall~ z_1,z_2,\cdots,z_N \in \mathbb{C},~ \forall~ \mathbf{m}_1,\mathbf{m}_2, \cdots, \mathbf{m}_N \in \mathbb{Z}^d,$ and $N$ a positive integer, we have
\be\begin{split}
\sum_{i=1}^N \sum_{j=1}^N \overline{z}_i z_j \gamma_x ( \mathbf{m}_i - \mathbf{m}_j )&=\sum_{i=1}^N \sum_{j=1}^N \overline{z}_i z_j ( \mathbf{T}^{\mathbf{m}_i} x , \mathbf{T}^{\mathbf{m}_j} x )_{\tau}\\
&=( \sum_{i=1}^N z_i \mathbf{T}^{\mathbf{m}_i} x , \sum_{j=1}^N z_j \mathbf{T}^{\mathbf{m}_j} x )_{\tau}\geq0,
\end{split}\ee
which implies the positive definiteness by definition. The existence of a spectral measure comes from Bochner's theorem.
\end{proof}

Let $x \in L_2(\mathcal{M})$. For each $\mathbf{n}\in \mathbb{N}^d $, we define
the $L^+_1(\M)$-valued function associated with $x$ as $ q_x^{\mathbf{n}}(\mathbf{z}) = \frac{1}{|\mathbf{n+1}|} \left| \sum_{\mathbf{k}=\mathbf{0}}^{\mathbf{n}} \mathbf{T}^{\mathbf{k}} (x) {\mathbf{z}}^{\mathbf{k}} \right|^2$, which is obviously continuous on $\mathbb{T}^d$.
Consequently it yields a $L^+_1(\M)$-valued Borel measure $\sigma_{x}^{\mathbf{n}}$ on $\mathbb{T}^d$ by
\be
\sigma_x^{\mathbf{n}}(E)= \int_E q_x^{\mathbf{n}}(\mathbf{z}) d\mathbf{z}, \ \ \ \ \ \ E \subset \mathbb{T}^d \text{ any Borel set.}
\ee
As usual, we will write $d \sigma_x^{\mathbf{n}}(\mathbf{z})= q_x^{\mathbf{n}}(\mathbf{z}) d \mathbf{z}$.

\section{proof of Theorem \ref{th:Main 1}}
As usual, we will deduce the desired bilateral almost uniform convergence for all operators in $L_1\log^{2{d-1}}L$ using the Banach principle. That is, we will first prove the bilateral almost uniform convergence for operators from a dense subspace of $L_1\log^{2{d-1}}L$, then establish maximal inequality on the space $L_1\log^{2{d-1}}L$, which will allow us to apply the Banach principle. In the present case, the Banach principle we need is slightly different from the previous known ones, so we formulate it in Lemma \ref{lem:densetoall}, and Theorem \ref{th:Main 1} follows immediately. Let us start with the a.u convergence for operators in a dense subset.

\subsection{Almost uniform convergence for operators from a dense subset} In the Wiener-Wintner type theorem, this part is considered as the most difficult part of the whole proof. For a given ergodic dynamical system$(\M,\tau,\mathbf{T})$, we will make use of the Kronecker factor $\mathcal{K}=\overline{E}^{\|\cdot\|_2}$, which has been identified in the statement of Theorem \ref{th:Main  2} in the Introduction, and break the operator in $L_2(\M)$ in order to prove the a.u. convergence independently for the eigenfunctions and those functions in the orthocomplement of the Kronecker factor. Due to the noncommutativity, the dense subset of $L_1\log^{2{d-1}}L(\M)$ that we construct is $L_\8(\M)\cap (E\oplus \mathcal{K}^{\perp})$. Yet we need a lemma in advance which concerns the weak convergence of operator valued measure sequence $\{\sigma_x^{\mathbf{n}}\}$.

\begin{lem}\label{lem:noncommuweakconverge} Given an ergodic dynamical system $(\M, \tau, \mathbf{T})$ and let $x \in L_{\8}(\mathcal{M})\subset L_2(\mathcal{M})$. Then for any $f\in \mathcal{C}(\mathbb{T}^d)$ (continuous complex functions on $d$ torus), we have
$$\left\{\int_{\mathbb{T}^d} f(\mathbf{z}) d \sigma_x^{\mathbf{n}}(\mathbf{z}) \right\}\text{ converges b.a.u. to}  \int_{\mathbb{T}^d} f(\mathbf{z}) d \sigma_x(\mathbf{z}) \cdot \mathbb{I}.$$
\end{lem}

\begin{proof}
Step 1. We calculate the Fourier-Stieltjes transform of $\sigma_x^{\mathbf{n}}$.
\be\begin{split}
\widehat{\sigma_x^{\mathbf{n}}}(\mathbf{m}) & = \int_{\mathbb{T}^d}\mathbf{z}^{\mathbf{m}}d\sigma_x^{\mathbf{n}}(\mathbf{z})\\
& = \int_{\mathbb{T}^d}\mathbf{z}^{\mathbf{m}} \frac{1}{|\mathbf{n+1}|} \left| \sum_{\mathbf{k}=\mathbf{0}}^{\mathbf{n}} \mathbf{T}^{\mathbf{k}} (x) {\mathbf{z}}^{\mathbf{k}} \right|^2 d \mathbf{z}\\
& = \frac{1}{|\mathbf{n+1}|} \int_{\mathbb{T}^d} \Bigg( \sum_{\mathbf{j}=\mathbf{0}}^{\mathbf{n}} \mathbf{T}^{\mathbf{j}} (x^{\ast}) \overline{\mathbf{z}}^{\mathbf{j}} \Bigg) \Bigg( \sum_{\mathbf{k}=\mathbf{0}}^{\mathbf{n}} \mathbf{T}^{\mathbf{k}} (x) {\mathbf{z}}^{\mathbf{k}} \Bigg)
\mathbf{z}^{\mathbf{m}} d \mathbf{z}\\
& = \frac{1}{|\mathbf{n+1}|} \int_{\mathbb{T}^d} \sum_{\mathbf{j}=\mathbf{0}}^{\mathbf{n}} \sum_{\mathbf{k}=\mathbf{0}}^{\mathbf{n}}
\mathbf{T}^{\mathbf{j}} (x^{\ast}) \mathbf{T}^{\mathbf{k}} (x) \mathbf{z}^{\mathbf{m}+\mathbf{k}-\mathbf{j}} d \mathbf{z}\\
& = \frac{1}{|\mathbf{n+1}|} \int_{\mathbb{T}^d} \sum_{\mathbf{k}=\mathbf{0}}^{\mathbf{n}} \sum_{\mathbf{i}=\mathbf{-k}}^{\mathbf{n-k}}
\mathbf{T}^{\mathbf{k+i}} (x^{\ast}) \mathbf{T}^{\mathbf{k}} (x) \mathbf{z}^{\mathbf{m}-\mathbf{i}} d \mathbf{z}\\
& = \frac{1}{|\mathbf{n+1}|} \sum_{\mathbf{0} \le \mathbf{k} \le \mathbf{n} \atop \mathbf{-m} \le \mathbf{k} \le \mathbf{n-m} }
\mathbf{T}^{\mathbf{k+ m}} (x^{\ast}) \mathbf{T}^{\mathbf{k}} (x).
\end{split}\ee
Here, we only deal with the case $m_1 \le 0 , m_2, \cdots, m_d \ge 0$, and the other cases are of the same nature. We calculate the correlation in the following,
\be\begin{split}
\widehat{\sigma_x^{\mathbf{n}}}(\mathbf{m})& = \frac{1}{|\mathbf{n+1}|} \sum_{ k_1= -m_1 }^{ n_1 } \sum_{ k_2 = 0}^{ n_2 -m_2  } \cdots \sum_{ k_d = 0}^{ n_d -m_d } \mathbf{T}^{\mathbf{k+m}} (x^{\ast}) \mathbf{T}^{\mathbf{k}} (x) \\
& = \frac{1}{|\mathbf{n+1}|} \sum_{ j_1= 0 }^{ n_1-|m_1| } \sum_{ j_2 = 0}^{ n_2 -m_2 } \cdots \sum_{ j_d = 0}^{ n_d -m_d } T_1^{j_1}T_2^{j_2+m_2}\cdots
T_d^{j_d+m_d} (x^{\ast})~ T_1^{j_1+|m_1|}T_2^{j_2}\cdots T_d^{j_d} (x)\\
& = \frac{1}{|\mathbf{n+1}|} \sum_{ j_1= 0 }^{ n_1-|m_1| } \sum_{ j_2 = 0}^{ n_2 -m_2 } \cdots \sum_{ j_d = 0}^{ n_d -m_d } \mathbf{T}^{\mathbf{j}}\Bigg(
T_2^{m_2}\cdots T_d^{m_d}(x^{\ast})~ T_1^{|m_1|}(x)\Bigg).
\end{split}\ee
Thus by the multi-parameter individual ergodic theorem \cite{JX2007}, $\widehat{\sigma_x^{\mathbf{n}}}(\mathbf{m})$ a.u. converges to
\be\begin{split}
\tau\Bigg( T_2^{m_2}\cdots T_d^{m_d}(x^{\ast})~ T_1^{|m_1|}(x)\Bigg) \cdot \mathbb{I} = \gamma_x(\mathbf{m})\cdot \mathbb{I}=\widehat{\sigma_x}(\mathbf{m})\cdot \mathbb{I}.
\end{split}\ee

We restate the previous conclusion: Given $ \mathbf{m}\in \mathbb{Z}^d $, for any $\ep > 0$, there exists $e_{\mathbf{m}} \in P(\M)$, such that$$ \tau(e_{\mathbf{m}}^{\bot}) < \frac{1}{3^d}\cdot\frac{\ep }{ 2^{|m_1|\cdot|m_2| \cdots |m_d|}} $$ and $$ \lim_{\mathbf{n}} \bigg\|  \big[ \widehat{\sigma_x^{\mathbf{n}}}(\mathbf{m}) - \widehat{\sigma_x}(\mathbf{m})\cdot \mathbb{I} \big] e_{\mathbf{m}} \bigg\|_{\8} = 0 .$$
Take $e=\wedge_{\mathbf{m}=-\8}^{+\8}e_{\mathbf{m}}$, then
\be\begin{split}
\tau(e^{\bot}) & < \lim_{\mathbf{N}}\sum_{\mathbf{m}=-\mathbf{N}}^{\mathbf{N}}\frac{1}{3^d}\cdot\frac{\ep }{ 2^{|m_1|\cdot|m_2| \cdots |m_d|}}\\
& = \lim_{\mathbf{N}} \sum_{m_1=-N_1}^{N_1} \sum_{m_2=-N_2}^{N_2} \cdots \sum_{m_d=-N_d}^{N_d}\frac{1}{3^d}\cdot\frac{\ep }{ 2^{|m_1|\cdot|m_2| \cdots |m_d|}}\\
& = \ep
\end{split}\ee
and for all $ \mathbf{m}\in \mathbb{Z}^d $,
\begin{align}\label{identity}
 \lim_{\mathbf{n}} \bigg\|  \big[ \widehat{\sigma_x^{\mathbf{n}}}(\mathbf{m}) - \widehat{\sigma_x}(\mathbf{m})\cdot \mathbb{I} \big] e \bigg\|_{\8} = 0 .
\end{align}

Step 2.
From Step 1, we know immediately that the desired result holds when $f$ is a trigonometric polynomial.
For general continuous function, we proceed with an approximation argument.

Fix $f\in \mathcal{C}(\mathbb{T}^{d})$, by Weierstrass approximation theorem we can find a sequence of polynomials $\{\varphi_l\}_{l=0}^\8$ uniformly converges to $f.$ On the other hand, from equality (\ref{identity}), we know that
$$\sup_{\mathbf{n}}\|e\sigma_x^{\mathbf{n}}(\mathbb{T}^d)e\|_\infty=\sup_{\mathbf{n}}\|e\widehat{\sigma_x^{\mathbf{n}}}(\mathbf{0})e\|_\infty<\infty.$$
Note that for any complex valued continuous function $g$ defined on $\mathbb{T}^d$, we have
\begin{align*}
\bigg\|e\int_{\mathbb{T}^d} g(\mathbf{z}) d \sigma_x^{\mathbf{n}}(\mathbf{z})e\bigg\|_\infty&\leq \bigg\|e\int_{\mathbb{T}^d} Re[g(\mathbf{z})] d \sigma_x^{\mathbf{n}}(\mathbf{z})e\bigg\|_\infty+\bigg\|e\int_{\mathbb{T}^d} Im[g(\mathbf{z})] d \sigma_x^{\mathbf{n}}(\mathbf{z})e\bigg\|_\infty\\
&\leq 2\bigg\|e\int_{\mathbb{T}^d} |g(\mathbf{z})|d \sigma_x^{\mathbf{n}}(\mathbf{z})e\bigg\|_\infty\leq 2\|g\|_{L_{\infty}(\mathbb{T}^d)}\sup_{\mathbf{n}}\|e\sigma_x^{\mathbf{n}}(\mathbb{T}^d)e\|_\infty.
\end{align*}
Hence for any $\delta>0,$ there exists $l_0$ and $\mathbf{n}_0$ such that when $l>l_0$ and $\mathbf{n}>\mathbf{n}_0$, we have
$$\bigg\| e\bigg[ \int_{\mathbb{T}^d} f(\mathbf{z}) d \sigma_x^{\mathbf{n}}(\mathbf{z})-\int_{\mathbb{T}^d} \varphi_l(\mathbf{z}) d \sigma_x^{\mathbf{n}}(\mathbf{z}) \bigg] e \bigg\|_{\8} < \delta\slash 3,$$
$$\bigg\|e \bigg[ \int_{\mathbb{T}^d} \varphi_l(\mathbf{z}) d \sigma_x^{\mathbf{n}}(\mathbf{z})-\int_{\mathbb{T}^d} \varphi_l(\mathbf{z}) d \sigma_x(\mathbf{z}) \bigg] e \bigg\|_{\8} < \delta\slash 3$$
and
$$\bigg\| e \bigg[ \int_{\mathbb{T}^d} f(\mathbf{z}) d \sigma_x(\mathbf{z})-\int_{\mathbb{T}^d} \varphi_l(\mathbf{z}) d \sigma_x(\mathbf{z}) \bigg] e \bigg\|_{\8} < \delta\slash 3.$$
Then it follows that,
\be\begin{split}
& \bigg\| \bigg[ \int_{\mathbb{T}^d} f(\mathbf{z}) d \sigma_x^{\mathbf{n}}(\mathbf{z})-\int_{\mathbb{T}^d} f(\mathbf{z}) d \sigma_x(\mathbf{z}) \bigg] e \bigg\|_{\8} \\
 \le & \bigg\|  \bigg[ \int_{\mathbb{T}^d} f(\mathbf{z}) d \sigma_x^{\mathbf{n}}(\mathbf{z})-\int_{\mathbb{T}^d} \varphi_l(\mathbf{z}) d \sigma_x^{\mathbf{n}}(\mathbf{z})\bigg] e \bigg\|_{\8} + \bigg\|  \bigg[ \int_{\mathbb{T}^d} \varphi_l(\mathbf{z}) d \sigma_x^{\mathbf{n}}(\mathbf{z})-\int_{\mathbb{T}^d} \varphi_l(\mathbf{z}) d \sigma_x(\mathbf{z})\bigg] e \bigg\|_{\8}\\
  + & \bigg\|  \bigg[ \int_{\mathbb{T}^d} \varphi_l(\mathbf{z}) d \sigma_x(\mathbf{z})-\int_{\mathbb{T}^d} f(\mathbf{z}) d \sigma_x(\mathbf{z})\bigg] e \bigg\|_{\8}< \delta,
\end{split}\ee
for any  $l>l_0$ and $\mathbf{n}>\mathbf{n}_0$. Thus we complete the proof.
\end{proof}

\begin{prop}\label{prop:continuousmeasure}
If $x\in \mathcal{K}^{\bot}$, then $\sigma_x$ is a continuous measure.
\end{prop}

\begin{proof}
We only prove the case $d=2$, and other cases follow from the similar argument. It is sufficient to show that $\sigma_x(\{ (z_1,z_2) \})=0$ for every $(z_1,z_2) \in \mathbb{T}^2 $. Using a two dimensional version of Theorem 7.13 \cite{Katz1976}, we have

\be\begin{split}
\sigma_x(\{ (z_1,z_2) \}) = &\lim_{n_1,n_2}\frac{1}{(2n_1 +1)(2n_2 +1)} \sum_{k_1=-n_1}^{n_1}\sum_{k_2=-n_2}^{n_2}z_1^{k_1}z_2^{k_2}\widehat{\sigma_x}(k_1,k_2)\\
 = & \lim_{n_1,n_2}\frac{1}{(2n_1 +1) (2n_2 +1)} \sum_{k_1=-n_1}^{n_1}\sum_{k_2=-n_2}^{n_2}z_1^{k_1}z_2^{k_2}\tau\left(x^{\ast}T_1^{k_1}T_2^{k_2}(x)\right)
\\
= & \lim_{n_1,n_2} \tau\left(x^{\ast}\ \frac{1}{(2n_1 +1) (2n_2 +1)}\sum_{k_1=-n_1}^{n_1} \sum_{k_2=-n_2}^{n_2} z_1^{k_1}T_1^{k_1} z_2^{k_2}T_2^{k_2}(x)\right).
\end{split}\ee
Then using mean ergodic theorem, we get that
\be
\frac{1}{(2n_1 +1) (2n_2 +1)} \sum_{k_1=-n_1}^{n_1} \sum_{k_2=-n_2}^{n_2} z_1^{k_1}T_1^{k_1} z_2^{k_2}T_2^{k_2}(x)
\ee
converges in $L_2$ sense to some element in $\mathcal{K}$.
On the other hand, by the fact that $T_i(\mathcal K^\perp)\subset\mathcal{K}^\perp$, we get
\be
\frac{1}{(2n_1 +1) (2n_2 +1)} \sum_{k_1=-n_1}^{n_1} \sum_{k_2=-n_2}^{n_2} z_1^{k_1}T_1^{k_1} z_2^{k_2}T_2^{k_2}(x) \in \mathcal{K}^{\bot}
\ee for all $n_1,n_2,$
thus the $L_2$-limit of
\be
\frac{1}{(2n_1 +1) (2n_2 +1)} \sum_{k_1=-n_1}^{n_1} \sum_{k_2=-n_2}^{n_2} z_1^{k_1}T_1^{k_1} z_2^{k_2}T_2^{k_2}(x)
\ee
is an element in $\mathcal{K}^{\perp}$.
Therefore $\sigma_x(\{ (z_1,z_2) \})$ is equal to 0 for any $(z_1,z_2) \in \mathbb{T}^2$, and we complete the proof.
\end{proof}

We introduce the definition of $\mathcal{D}$ here, and its relationship with bounded Besicovitch class will be shown in Section 5.

\begin{definition}\label{def:D}
We define $\mathcal{D}$ to be the set of all $\mathbf{a}\in S\cap \ell_\infty(\mathbb{N}^d)$ satisfying the following conditions:
\begin{enumerate}[{\rm (1)}]
\item The spectral measure $\sigma_{\mathbf{a}}$ corresponding to $\mathbf{a}$ is discrete;
\item The amplitude $ \lim_{\mathbf{n}}\frac{1}{|\mathbf{n+1}|} \sum_{\mathbf{k}=0}^{\mathbf{n}} a(\mathbf{k}) \mathbf{z}^{\mathbf{k}} $ exists for all $\mathbf{z}\in \mathbb{T}^d.$
\end{enumerate}
\end{definition}

\begin{prop}\label{pr:dense convergence}
Let $(\M,\tau,\mathbf{T})$ be a finite trace preserving ergodic dynamical system. Given $x\in L_\infty(\M)\cap (E\oplus\mathcal{K}^{\perp})$, for any $\e > 0$, there exists $e\in P(\M)$ such that
$$ \tau(e^\bot)\le \e \text{ and }  \quad \{M_{\mathbf{n}}(x,a)e \}$$
converges in $\M$ for all $\mathbf{a} \in \mathcal{D}$.
\end{prop}

\begin{proof}
Fix $x\in L_\infty(\M)\cap (E\oplus\mathcal{K}^{\perp})$. We separate the proof into two parts.

Part 1. Let $x\in E\cap L_{\8}(\mathcal{M}),$ by linear combination, we can assume $x\in\bigcap_{i=1}^d E_i$. Then the result follows trivially from condition (2) of Definition \ref{def:D} since
\be\begin{split}
 M_{\mathbf{n}}(x, \mathbf{a})& = \frac{1}{|\mathbf{n+1}|} \sum_{\mathbf{k}=\mathbf{0}}^{\mathbf{n}}a(\mathbf{k}) \mathbf{T}^{\mathbf{k}}(x)\\
 &= \frac{1}{|\mathbf{n+1}|}\sum_{\mathbf{k}=\mathbf{0}}^{\mathbf{n}}a(\mathbf{k})\underline{\mu}^{\mathbf{k}}x, ~\ \underline{\mu}\in \mathbb{T}^d.
\end{split}\ee

Part 2. Now let $x \in \mathcal{K}^{\bot} \cap L_{\8}(\mathcal{M})$. We first give a construction of partitions of $\mathbb{T}^d$ associated with spectral measure. This construction is first used by Coquet \cite{CKF1977} in realizing approximation process of the affinity between measures.

Let $\sigma_{\mathbf{a}}, \sigma_{\mathbf{a}}^{\mathbf{n}}, \sigma_x ,\sigma_x^{\mathbf{n}}$ be the corresponding positive measures we defined in Section 2.2. Let $\nu=\sigma_{\mathbf{a}}+\sigma_x$, so that $\sigma_{\mathbf{a}}$ and $\sigma_x$ are absolutely continuous with respect to $\nu$.
For any $\delta >0$, we define the following sets,
\be
\Omega=\left\{ \mathbf{z}\in \mathbb{T}^d: \frac{d\sigma_{\mathbf{a}}}{d\nu}(\mathbf{z})=0 \right\},
\ee
\be
U_l=\left\{ \mathbf{z}\in \mathbb{T}^d \backslash \Omega: (1+\delta)^l \le \frac{(d\sigma_x\slash d\nu) (\mathbf{z})}{(d\sigma_{\mathbf{a}}\slash d\nu)(\mathbf{z})}<(1+\delta)^{l+1} \right\}, \ \ l\in \mathbb{Z}.
\ee
The series $\sum_l \sigma_{\mathbf{a}}(U_l)$ is convergent, so that $r(n)=\sum_{|l|\ge n} \sigma_{\mathbf{a}}(U_l)$ tends to zero when $n$ tends to infinity.

For any positive integer $N$, We consider the following sets:
\be
V_0 = \Omega \cup \bigcup_{|l|\ge N}U_l,
\ee
\be
V_1 = \left\{ \mathbf{z}\in \mathbb{T}^d \backslash \Omega: \frac{d \sigma_x}{d\nu}(\mathbf{z}) =0 \right\}
\ee
and
\be
V_l=U_{l-1-N}, ~\ l=2,3,\cdots, 2N.
\ee
The family $\{ V_0,V_1,\cdots,V_{2N} \}$ is a partition of $\mathbb{T}^d$, and $\sigma_{\mathbf{a}}(V_0)=r(N).$

Now we choose $N$ such that $\sigma_{\mathbf{a}}^{1 \slash 2}(V_0)\sigma_x^{1 \slash 2}(V_0)< \delta$, and it is obvious that $\sigma_{\mathbf{a}}^{1 \slash 2}(V_1)\sigma_x^{1 \slash 2}(V_1)=0.$
Under these conditions,
\beq\label{eq:affinity}\begin{split}
\rho(\sigma_{\mathbf{a}},\sigma_x) &\ge \sum_{l=2}^{2N} \int_{V_l}  \left(\frac{d \sigma_{\mathbf{a}}}{d\nu}\right)^{1\slash 2} \left(\frac{d \sigma_x}{d\nu}\right)^{1\slash 2}d\nu\\
& \ge \sum_{l=2}^{2N} (1+\delta)^{(l-1-N)\slash 2} \sigma_{\mathbf{a}}(V_l)\\
& \ge (1+\delta)^{-1 \slash 2} \sum_{l=2}^{2N} [\sigma_{\mathbf{a}}(V_l)]^{1\slash 2} [\sigma_x(V_l)]^{1\slash 2}\\
& \ge (1+\delta)^{-1 \slash 2} \sum_{l=0}^{2N} [\sigma_{\mathbf{a}}(V_l)]^{1\slash 2} [\sigma_x(V_l)]^{1\slash 2}-\delta.
\end{split}\eeq

So we can find a family of continuous non-negative functions $f_l: \mathbb{T}^d \rightarrow \mathbb{R}$ $(l=0,1,\dots,2N)$ which satisfies:
\begin{enumerate}[{\rm (i)}]
\item $\sum_{l=0}^{2N}f_l=1$;
\item $\begin{array}{lcl}
\int_{\mathbb{T}^d}f_l d\sigma_{\mathbf{a}}  \le  (1+\delta)^{1\slash 2} \sigma_{\mathbf{a}}(V_l),\\
\int_{\mathbb{T}^d}f_l d\sigma_{x}  \le  (1+\delta)^{1\slash 2} \sigma_{x}(V_l), \ \ l=0,1,\dots,2N.
      \end{array}$
\end{enumerate}
Then from (\ref{eq:affinity}) we get:
\be
\rho(\sigma_{\mathbf{a}},\sigma_x) \ge (1+ \delta)^{-1}\sum_{l=0}^{2N}\left( \int_{\mathbb{T}^d}f_l d\sigma_{\mathbf{a}} \right)^{1 \slash 2}\left( \int_{\mathbb{T}^d}f_l d\sigma_{x} \right)^{1 \slash 2}-\delta.
\ee
From condition (1) of Definition \ref{def:D} and Proposition \ref{prop:continuousmeasure}, we know $\sigma_{\mathbf{a}}$ and $\sigma_x$ are respectively discrete and continuous thus mutually singular, which means the affinity
$\rho(\sigma_{\mathbf{a}},\sigma_x)= 0.$  Then with $\delta$ arbitrary, we have
\beq\label{eq:affinity2}
\sum_{l=0}^{2N}\left( \int_{\mathbb{T}^d}f_l d\sigma_{\mathbf{a}} \right)^{1 \slash 2}\left( \int_{\mathbb{T}^d}f_l d\sigma_{x} \right)^{1 \slash 2}=0.
\eeq

Now we calculate the following for some projection $e$ which will be fixed later:
\be\begin{split}
\left\| M_{\mathbf{n}}(x,\mathbf{a}) e\right\|_\8 & = \left\|\frac{1}{|\mathbf{n+1}|} \sum_{\mathbf{k}=\mathbf{0}}^{\mathbf{n}}a(\mathbf{k}) \mathbf{T}^{\mathbf{k}}(x)e\right\|_\8\\
& = \left\| \frac{1}{|\mathbf{n+1}|} \int_{\mathbb{T}^d}\sum_{\mathbf{k}=\mathbf{0}}^{\mathbf{n}}\sum_{\mathbf{j}=\mathbf{0}}^{\mathbf{n}}
a(\mathbf{j})  \mathbf{T}^{\mathbf{k}}(x)e~\mathbf{z}^{\mathbf{k}-\mathbf{j}}d\mathbf{z}\right\|_\8\\
& = \left\|\int_{\mathbb{T}^d} \Bigg(\frac{1}{\sqrt{|\mathbf{n+1}|}}\sum_{\mathbf{j}=\mathbf{0}}^{\mathbf{n}}a(\mathbf{j})\mathbf{z}^{-\mathbf{j}} \Bigg)
\Bigg( \frac{1}{\sqrt{|\mathbf{n+1}|}}\sum_{\mathbf{k}=\mathbf{0}}^{\mathbf{n}} \mathbf{T}^{\mathbf{k}}(x)\mathbf{z}^{\mathbf{k}}\Bigg)e~ d\mathbf{z}\right\|_\8\\
& = \left\|\sum_{l=0}^{2N}\int_{\mathbb{T}^d} f_l(\mathbf{z}) \Bigg(\frac{1}{\sqrt{|\mathbf{n+1}|}}\sum_{\mathbf{j}=\mathbf{0}}^{\mathbf{n}}a(\mathbf{j})\mathbf{z}^{-\mathbf{j}} \Bigg)
\Bigg( \frac{1}{\sqrt{|\mathbf{n+1}|}}\sum_{\mathbf{k}=\mathbf{0}}^{\mathbf{n}} \mathbf{T}^{\mathbf{k}}(x)\mathbf{z}^{\mathbf{k}}\Bigg)e~ d\mathbf{z}\right\|_\8.
\end{split}\ee
Then by triangle inequality and Kadison-Schwarz inequality, we have
\be\begin{split}
\left\| M_{\mathbf{n}}(x,\mathbf{a}) e\right\|_\8 & \le \sum_{l=0}^{2N} \left\| \int_{\mathbb{T}^d} f_l(\mathbf{z}) \Bigg(\frac{1}{\sqrt{|\mathbf{n+1}|}}\sum_{\mathbf{j}=\mathbf{0}}^{\mathbf{n}}a(\mathbf{j})\mathbf{z}^{-\mathbf{j}} \Bigg)
\Bigg( \frac{1}{\sqrt{|\mathbf{n+1}|}}\sum_{\mathbf{k}=\mathbf{0}}^{\mathbf{n}} \mathbf{T}^{\mathbf{k}}(x)\mathbf{z}^{\mathbf{k}}\Bigg)e~ d\mathbf{z}\right\|_\8\\
& \le \sum_{l=0}^{2N}\left\| \left[ \int_{\mathbb{T}^d} \frac{f_l(\mathbf{z})}{|\mathbf{n+1}|}\Bigg|\sum_{\mathbf{j}=\mathbf{0}}^{\mathbf{n}}a(\mathbf{j})\mathbf{z}^{-\mathbf{j}} \Bigg|^2d\mathbf{z}\right]^{\frac{1}{2}}
\left[ \int_{\mathbb{T}^d} e \frac{f_l(\mathbf{z})}{|\mathbf{n+1}|}\Bigg|\sum_{\mathbf{k}=\mathbf{0}}^{\mathbf{n}} \mathbf{T}^{\mathbf{k}}(x)\mathbf{z}^{\mathbf{k}} \Bigg|^2 e~ d\mathbf{z}\right]^{\frac{1}{2}} \right\|_\8\\
& = \sum_{l=0}^{2N}\left\| \left(\int_{\mathbb{T}^d}f_l d \sigma_{\mathbf{a}}^{\mathbf{n}} \right)^{1 \slash 2} \left(e \int_{\mathbb{T}^d} f_l d \sigma_{x}^{\mathbf{n}}~ e \right)^{1 \slash 2}\right\|_\8.
\end{split}\ee

Apply Lemma \ref{lem:noncommuweakconverge} to $x$: For any $\varepsilon>0$, there exists a projection $e$ such that
$\tau(e^\perp)<\varepsilon$ and
$$\left\|e \left(\int_{\mathbb{T}^d} f_l d \sigma_{x}^{\mathbf{n}}-\int_{\mathbb{T}^d} f_l d \sigma_{x}\right)  e\right\|_\8$$
converges to 0. Now that $e$ is fixed, by Proposition \ref{prop:measureconvergence}

\be\begin{split}
\left\| M_{\mathbf{n}}(x,\mathbf{a})e\right\|_\8 \le & \sum_{l=0}^{2N}\left\| \left(\int_{\mathbb{T}^d}f_l d \sigma_{\mathbf{a}}^{\mathbf{n}} \right)^{1 \slash 2} \left(e\int_{\mathbb{T}^d} f_l d \sigma_{x}^{\mathbf{n}}e \right)^{1 \slash 2}\right\|_\8
\end{split}\ee
converges to
\be\begin{split}
 \sum_{l=0}^{2N}\left\|\left( \int_{\mathbb{T}^d}f_l d\sigma_{\mathbf{a}} \right)^{1 \slash 2} \left( \int_{\mathbb{T}^d}f_l d\sigma_{x} \right)^{1 \slash 2}e\right\|_\8= & \sum_{l=0}^{2N}\left( \int_{\mathbb{T}^d}f_l d\sigma_{\mathbf{a}} \right)^{1 \slash 2} \left( \int_{\mathbb{T}^d}f_l d\sigma_{x} \right)^{1 \slash 2} = 0
\end{split}\ee
{by }(\ref{eq:affinity2}).
\end{proof}

\subsection{Maximal inequality} The maximal inequality in proving the classical one-parameter Wiener-Wintner's theorem follows from a trivial argument. The desired multi-parameter noncommutative maximal inequality requires new ideas since the classical extrapolation argument and the existed noncommutative one \cite{Hu2009} do not work. We make use of the atomic decomposition of Orlicz space discovered by Tao \cite{Tao2001}.

\begin{thm}\label{th:maximalergodic}
Let $(\M, \tau, \mathbf{T})$ be a finite trace preserving dynamical system and $\mathbf{a}\in \ell_{\8}(\mathbb{N}^d)$. Given $x\in L_1\log^{2(d-1)} L(\mathcal{M})$, for any $\lambda>0$ there are positive constant $C$ and a projection $e\in P(\mathcal{M})$ such that
$$\tau(e^{\bot}) \le
 C \frac{\| x \|_{L_1\log^{2(d-1)}L}}{\lambda} ~ \text{ and } ~ \sup_{\mathbf{n}}\|e (M_{\mathbf{n}}(x, \mathbf{a})) e\|_{\8} \le \| \mathbf{a}\|_{\8} \lambda.$$
Moreover, for $x\in L_2\log^{2(d-1)} L(\mathcal{M})$ we have the following estimates $$\tau(e^{\bot}) \le
 \left(C \frac{\| x \|_{L_2\log^{2(d-1)}L}}{\lambda}\right)^2 ~ \text{ and } ~ \sup_{\mathbf{n}}\| (M_{\mathbf{n}}(x, \mathbf{a})) e\|_{\8} \le \| \mathbf{a}\|_{\8} \lambda.$$
\end{thm}

\begin{proof}
We first prove a non-weighted version of the maximal inequalities for positive operators.
We start with the following one-parameter maximal inequality between Orlicz spaces, which is the key point of the whole arguments.
Let $T$ be any $T_i$, $i=1,\dotsm,d$. For any $s\geq0$, there exists constant $C$ such that
\be
 \|\{M_{n}(T)x\}_{n}\|_{L_1\log^{s} L(\ell_\8)}\le C \|x\|_{L_1\log^{s+2} L}
\ee
holds for any $x\in L_1\log^{s+2}L(\M).$

In [Tao01], the author establish atomic characterization of commutative Orlicz space $L_1\log^{r} L$, $r>0$. Starting with the generalized singular numbers, we note that the same arguments work well in the noncommutative setting. That is, every positive $x\in L_1\log^{s+2} L(\M)$ can be decomposed into convex linear combination of atoms:
$$x= \sum_e \lambda_e x_e$$
where $x_e$'s are atoms of the form $\frac{1}{\tau(e)}(\log(\frac{1}{\tau(e)}))^{-(s+2)}e$ with $e$ being projection as well as $0<\tau(e) \ll 1$, and $\|x\|_{L_1\log^{s+2}L}$ is equivalent to $\sum_e\lambda_e$. Thus we are reduced to the estimate over every atom
\be\begin{split}
 \|\{M_{n}(T)x\}_{n}\|_{L_1\log^{s} L(\ell_\8)} &= \| \{M_{n}(T)\sum_e \lambda_e x_e \|_{L_1\log^{s} L(\ell_\8)}\\
& \le  \sum_e \lambda_e\| \{M_{n}(T) x_e \|_{L_1\log^{s} L(\ell_\8)}.
\end{split}\ee
Now taking $p=1+ \big(\log\frac{1}{\tau(e)}\big)^{-\frac{s+2}{2}}$, using the fact that the ${L_p(\M;\ell_\8)}$ embeds into ${L_1\log^s L(\M;\ell_\8)}$ and applying Junge and Xu's ergodic maximal inequality \cite{JX2007}, we have
\be\begin{split}
& \|\{M_{n}(T)x\}_{n}\|_{L_1\log^{s} L(\ell_\8)} \leq \sum_e \lambda_e\| \{M_{n_i}(T_i) x_e \|_{L_p(\ell_\8)}\\
&\le C\sum_e \frac{\lambda_e}{(p-1)^2}\| x_e \|_{p} = C\sum_e \lambda_e \frac{1}{(p-1)^2} \frac{1}{\tau(e)} \left(\log\frac{1}{\tau(e)}\right)^{-(s+2)} \Big(\tau(e)\Big)^{\frac{1}{p}}\\
&\leq C \sum_e \lambda_e \Big(\tau(e)\Big)^{\frac{1}{p}-1}.
\end{split}\ee
The desired estimate follows from the fact that $\Big(\tau(e)\Big)^{\frac{1}{p}-1}\leq 1$.
Indeed, let $1 \ll \frac{1}{\tau(e)}= a$,
\be\begin{split}
\Big(\tau(e)\Big)^{\frac{1}{p}-1} & = \Big(\tau(e)\Big)^{\frac{1}{1+ \big(\log\frac{1}{\tau(e)}\big)^{-\frac{s+2}{2}}}-1}
 = \Big(\tau(e)\Big)^{\frac{-\big(\log\frac{1}{\tau(e)}\big)^{-\frac{s+2}{2}}}{1+ \big(\log\frac{1}{\tau(e)}\big)^{-\frac{s+2}{2}}}}\\
& = a^{\frac{\big(\log a \big)^{-\frac{s+2}{2}}}{1+ \big(\log a\big)^{-\frac{s+2}{2}}}}
 \le a^{\big(\log a \big)^{-\frac{s+2}{2}}}
 = e^{\log a^{\big(\log a \big)^{-\frac{s+2}{2}}}}\\
& = e^{\big(\log a\big)^{-\frac{s}{2}}}\leq1.
\end{split}\ee
So finally we get
\be
 \|\{M_{n}(T)x\}_{n}\|_{L_1\log^{s} L(\ell_\8)}\le C\sum_e \lambda_e \le C \|x\|_{L_1\log^{s+2} L}.
\ee
By Proposition 3.1 in \cite{BCO2013} we obtain en equivalent formulation: There exists an operator $a$ such that
$$M_{n}(T)x\leq a\quad\forall n\quad\mathrm{and}\quad\|a\|_{L_1\log^{s} L}\leq C\|x\|_{L_1\log^{s+2} L}.$$

Now starting with $x\in L_1\log^{2(d-1)} L(\M)$. Applying $d$ times the previous equivalent formulation, there exists an operator $x_{d-1}$ such that
$$M_{n_1,\dotsm,n_{d-1}}(T_1,\dotsm,T_{d-1})x\leq x_d\quad\forall (n_1,\dotsm,n_{d-1})\quad\mathrm{and}\quad\|x_{d-1}\|_{1}\leq C\|x\|_{L_1\log^{2(d-1)} L}.$$
Then by Yeadon's weak type $(1,1)$ maximal inequality, it implies for any $\lambda>0$, there is a projection $e\in P(\mathcal{M})$ such that $$\tau(e^{\bot}) \le C \frac{\| x \|_{L_1\log^{2(d-1)} L(\M)}}{\lambda}  ~ \text{ and } ~ \sup_{\mathbf{n}}\|e (M_{\mathbf{n}}(x)) e\|_{\8} \le \lambda,$$
with $C$ being a constant independent of $x$ and $\lambda$.

When $x\in L_2\log^{2(d-1)}L$, we apply the previous estimate to $|x|^2$. That is, for any $\eta>0$, there is a projection $e\in P(\mathcal{M})$ such that $$\tau(e^{\bot}) \le C \frac{\| |x|^2 \|_{L_1\log^{2(d-1)} L(\M)}}{\eta}  ~ \text{ and } ~ \sup_{\mathbf{n}}\|e (M_{\mathbf{n}}(|x|^2)) e\|_{\8} \le \eta$$
with $C$ being a constant independent of $x$ and $\eta$.
Then for any $\lambda>0$, taking $\eta=\lambda^2$, we get the desired estimate since on the one hand
$\||x|^2\|_{L_1\log^{2(d-1)}L}$ is equivalent to $\|x\|^2_{L_2\log^{2(d-1)}L}$, on the other hand by Kadison-Schwarz inequality
$$\|M_{\mathbf{n}}(x) e\|^2_\infty\leq \|eM_{\mathbf{n}}(|x|^2) e\|_\infty\leq \lambda^2.$$

\bigskip

With the previous non-weighted estimates, the following arguments allow us to complete the proof.
Given $x\in {L_1\log^{2(d-1)} L(\M)}$, we have $x = (x_1 - x_2)+ i(x_3 - x_4),$ where $x_j \in {L_1\log^{2(d-1)} L^+(\M)}$ and $\|x_j\|_{L_1\log^{2(d-1)} L(\M)} \le \|x\|_{L_1\log^{2(d-1)} L(\M)}$ for each $j=1,2,3,4.$
So using previous non-weighted result, for any $\lambda>0$, there exists $e_j \in P(\mathcal{M})$ satisfying
\be
\tau(e_j^{\bot}) \le
 C \frac{\| x_j \|_{L_1\log^{2(d-1)} L(\M)}}{\lambda }  ~ \text{ and } ~ \sup_{\mathbf{n}}\|e_j (M_{\mathbf{n}}(x_j)) e_j\|_{\8} \le \frac{\lambda}{8}.
\ee
with $C$ an absolute constant.
Take $e=\wedge_{j=1}^4 e_j$, we have
$$ \tau(e) \le  C \frac{\| x \|_{L_1\log^{2(d-1)} L(\M)}}{\lambda }$$
and
\be\begin{split}
&\sup_{\mathbf{n}} \|e (M_{\mathbf{n}}(x_j, \mathbf{a})) e\|_{\8} = \sup_{\mathbf{n}} \left\| e \left( \frac{1}{|\mathbf{n}|} \sum_{\mathbf{k}=\mathbf{0}}^{\mathbf{n}-\mathbf{1}}a(\mathbf{k}) \mathbf{T}^{\mathbf{k}}(x_j) \right) e \right\|_{\8}\\
& = \sup_{\mathbf{n}} \left\| e \left( \frac{1}{|\mathbf{n}|} \sum_{\mathbf{k}=\mathbf{0}}^{\mathbf{n}-\mathbf{1}}Re [a(\mathbf{k})] \mathbf{T}^{\mathbf{k}}(x_j) \right) e +  e \left( \frac{1}{|\mathbf{n}|} \sum_{\mathbf{k}=\mathbf{0}}^{\mathbf{n}-\mathbf{1}}i Im[a(\mathbf{k})] \mathbf{T}^{\mathbf{k}}(x_j) \right) e\right\|_{\8}\\
& \le \sup_{\mathbf{n}}\left\{ \left\| e \left( \frac{1}{|\mathbf{n}|} \sum_{\mathbf{k}=\mathbf{0}}^{\mathbf{n}-\mathbf{1}}Re [a(\mathbf{k})] \mathbf{T}^{\mathbf{k}}(x_j) \right) e \right\|_{\8} + \left\| e \left( \frac{1}{|\mathbf{n}|} \sum_{\mathbf{k}=\mathbf{0}}^{\mathbf{n}-\mathbf{1}}i Im[a(\mathbf{k})] \mathbf{T}^{\mathbf{k}}(x_j) \right) e\right\|_{\8} \right\}\\
& \le \sup_{\mathbf{n}}\sup_{0\le \mathbf{k}\le(\mathbf{n}-\mathbf{1})} \big| a(\mathbf{k})\big|\left\{ \left\| e \left( \frac{1}{|\mathbf{n}|} \sum_{\mathbf{k}=\mathbf{0}}^{\mathbf{n}-\mathbf{1}}\mathbf{T}^{\mathbf{k}}(x_j) \right) e \right\|_{\8} + \left\| e \left( \frac{1}{|\mathbf{n}|} \sum_{\mathbf{k}=\mathbf{0}}^{\mathbf{n}-\mathbf{1}} \mathbf{T}^{\mathbf{k}}(x_j) \right) e\right\|_{\8}\right\}\\
& \le \sup_{\mathbf{n}} 2 \|\mathbf{a}\|_{\8} \left\| e \left( \frac{1}{|\mathbf{n}|} \sum_{\mathbf{k}=\mathbf{0}}^{\mathbf{n}-\mathbf{1}} \mathbf{T}^{\mathbf{k}}(x_j) \right) e \right\|_{\8} \le \|\mathbf{a}\|_{\8} \frac{\lambda}4.
\end{split}\ee
Thus using triangle inequality again,
\be
\sup_{\mathbf{n}} \|e (M_{\mathbf{n}}(x, \mathbf{a})) e\|_{\8} \ee
\be
=\sup_{\mathbf{n}} \|e (M_{\mathbf{n}}(x_1, \mathbf{a})-M_{\mathbf{n}}(x_2, \mathbf{a})+i M_{\mathbf{n}}(x_3, \mathbf{a})-i M_{\mathbf{n}}(x_4, \mathbf{a})) e\|_{\8} \le \|\mathbf{a}\|_{\8} \lambda.
\ee
\end{proof}

\subsection{Banach principle}
In the following, we give a result which plays the role of the Banach Principle as in classical ergodic theory.

\begin{lem}\label{lem:densetoall}
If for any fixed finite trace preserving dynamical system $(\M,\tau,\mathbf{T})$, there exists a dense subset $\mathcal{B}$ of $L_1\log^{2(d-1)}L(\M)$ such that for $x\in \mathcal{B}$, for any $\e > 0$, there exists $e\in P(\M)$ such that
$$ \tau(e^\bot)\le \e \text{ and }  \{M_{\mathbf{n}}(x,a)e \}$$
converges in $\M$ for all $\mathbf{a} \in \mathcal{A}$.
Then  $\mathcal{A}$ is of $L_1\log^{2(d-1)}L(\M)$-NCbWW type and $L_2\log^{2(d-1)}L(\M)$-NCWW type.
\end{lem}

\begin{proof}
We first prove $\mathcal{A}$ is of $L_1\log^{2(d-1)}L(\M)$-NCbWW type. Fix a dynamical system $(\M,\tau,\mathbf{T})$. Taking any $x\in L_1\log^{2(d-1)}L(\M)$, any $\ep>0$ and any $\lambda>0,$ since $\mathcal B$ is dense in $L_1\log^{2(d-1)}L(\M),$ we can always find one $y \in \mathcal B$,  so that $\|x-y\|_{L_1\log^{2(d-1)}L(\M)} \le \frac{1}{2C}\lambda \ep$, $C$ is the positive constant from the application of the first maximal inequality in Theorem \ref{th:maximalergodic} to the element $x-y$: there exists a projection $e_1\in \mathcal{P}(\mathcal{M})$ such that $$\tau(e_1^{\bot}) \le
 C \frac{\| x-y \|_{L_1\log^{2(d-1)}L(\M)}}{\lambda}  \le  \frac{\ep}{ 2} ~ \text{ and } ~ \sup_{\mathbf{n}}\|e_1 (M_{\mathbf{n}}(x-y, \mathbf{a})) e_1\|_{\8} \le \| \mathbf{a}\|_{\8} \lambda.$$
On the other hand, by the assumption, there exists an projection $e_2$, such that $\tau(e_2^{\bot})\le \ep \slash 2$ and $ \{e_2 M_{\mathbf{n}}(y, \mathbf{a}) e_2 \} $ converges in $ \mathcal{M} $, which means that there exists $\mathbf{N} \in \mathbb{N}^d$, whenever $\mathbf{m,n}\ge \mathbf{N}$, we have $$ \| e_2 M_{\mathbf{m}}(y, \mathbf{a}) e_2 - e_2 M_{\mathbf{n}}(y, \mathbf{a}) e_2 \|_{\8} < \lambda.$$

Now take $e = e_1\wedge e_2$, then we have $\tau(e^{\bot})\le \ep$ and
\be
\| e M_{\mathbf{m}}(x, \mathbf{a}) e - e M_{\mathbf{n}}(x, \mathbf{a}) e \|_{\8}
\ee
\be\begin{split}
\le & \| e M_{\mathbf{m}}(x-y, \mathbf{a}) e \|_{\8} +\| e M_{\mathbf{m}}(y, \mathbf{a}) e - e M_{\mathbf{n}}(y, \mathbf{a}) e \|_{\8} + \|e M_{\mathbf{n}}(y-x, \mathbf{a}) e\|_{\8}\\
< & (2 \|\mathbf{a} \|_{\8}+1)\lambda,
\end{split}\ee
which means $\{e M_{\mathbf{n}}(x, \mathbf{a}) e\}$ is Cauchy sequence, thus converges in $\mathcal{M}$ for all $\mathbf{a}\in \mathcal{A}$. Therefore, we conclude that $\mathcal{A}$ is of $L_1\log^{2(d-1)}L$-NCbWW type.

The fact that $\mathcal{A}$ is of $L_2\log^{2(d-1)}L$-NCWW type can be shown using a similar argument. The only difference is that we use the second maximal inequality in Theorem \ref{th:maximalergodic}. We omit the details.
\end{proof}

After we finish all the preparing work, we can now conclude the proof of Theorem \ref{th:Main 1}.
\begin{proof}[Proof of Theorem \ref{th:Main 1}]
Fix any dynamical system $(\M,\tau,\mathbf{T})$. Using decomposition into ergodic parts, we can assume the dynamical system is ergodic. By Proposition \ref{pr:dense convergence}, we can take $\mathcal{B}=L_\infty(\M)\cap (E\oplus\mathcal{K}^{\perp})$ in Lemma \ref{lem:densetoall}, together with Theorem \ref{th:maximalergodic}, we get the conclusion of Theorem \ref{th:Main 1}.\end{proof}

Finally, we introduce the concept of convergence in measure (see for instance \cite{CLS2005}) and give a quick proof of Corollary \ref{cor:NCWW} to finish this section. We say that a sequence $\{x_{\mathbf{n}}\}\in L_0(\M)$ converges in measure (resp. bilaterally in measure) to $x \in L_0(\M)$, if, for any $ \e > 0, \d > 0$, we can find $\mathbf{N} = \mathbf{N}(\e, \d)\in \mathbb{N}^d$ such that for every $\mathbf{n} > \mathbf{N}$, there exists $e_{\mathbf{n}} \in P(\M)$ satisfying $\tau(e_{\mathbf{n}}^\bot) \le \e$ and $\|(x_{\mathbf{n}} - x)e_{\mathbf{n}}\| < \d$ (resp. $\|e_{\mathbf{n}}(x_{\mathbf{n}} - x)e_{\mathbf{n}}\| < \d$). However, by a multi-parameter version of Theorem 2.2 \cite{CLS2005}, it is known that convergence in measure is equivalent to convergence bilaterally in measure, thus we treat the two notions as the same.

\begin{proof}[Proof of Corollary \ref{cor:NCWW}]
From Theorem \ref{th:Main 1}, we know for any finite trace preserving dynamical system $(\M,\tau,\mathbf{T})$, for any $x\in L_1\log^{2(d-1)}L(\M)$ and any $\e>0$, there exists $e\in P(\M)$ such that $\tau(e^\bot)\le \e$ and $\{eM_{\mathbf{n}}(x, \mathbf{a})e\}$ is a Cauchy sequence in $\M$ for all $\mathbf{a}\in \mathcal{D}$. Therefore, fix any $\mathbf{a}$, there exists $x_{\mathbf{a},e}\in \M$ such that $eM_{\mathbf{n}}(x, \mathbf{a})e \rightarrow x_{\mathbf{a},e}$ in $\M$. Hence $\{eM_{\mathbf{n}}(x, \mathbf{a})e\}$ converges to $x_{\mathbf{a},e}$ in measure. On the other hand, it is easy to see that $\{M_{\mathbf{n}}(x, \mathbf{a})\}$ is also a Cauchy sequence relative to the convergence in measure. Then from the completeness of $L_0(\M)$ with respect to the measure topology, one can find $F(x,\mathbf{a})\in L_0(\M)$ such that $\{M_{\mathbf{n}}(x, \mathbf{a})\}$ converges to $F(x,\mathbf{a})$ in measure, which implies that $\{eM_{\mathbf{n}}(x, \mathbf{a})e\}$ converges to $eF(x,\mathbf{a})e$ in measure. Thus we get $x_{\mathbf{a},e} = eF(x,\mathbf{a})e$. And the fact $F(x,\mathbf{a})\in L_1\log^{2(d-1)}L(\M)$ follows from a multi-parameter version of Theorem 1.2 \cite{CLS2005} since it is clear that $\|M_{\mathbf{n}}(x, \mathbf{a})\|_{L_1\log^{2(d-1)}L} \le \|\mathbf{a}\|_{\8} \|x\|_{L_1\log^{2(d-1)}L} <\8$.

For the case of $x\in L_2\log^{2(d-1)}L(\M)$, similar arguments are applicable and thus we complete the proof.
\end{proof}

\section{proof of Theorem \ref{th:Main 2}}
To simplify notation, we only prove the two-parameter case, and similar arguments work also for other cases.
As explained in the Introduction, we need a noncommutative Van der Corput's inequality in multi-parameter case, which should be regarded as a multi-parameter analogue of the one established in \cite{NSZ2005}.

\begin{lem}\label{le:VanderCorput}
 Let $1 \le n_i \ge h_i \ge 0~(i=1,2)$ are natural numbers and $\{a_{j_1,j_2}\}_{1\le j_1 \le n_1,1\le j_2 \le n_2}$
are elements of a $C^\ast$-algebra with the norm $\|\cdot\|$, denote $H=(h_1+1)(h_2 +1)$, then
\be\begin{split}
    & \left\| \frac{1}{n_1 n_2} \sum_{j_1=1}^{n_1} \sum_{j_2=1}^{n_2} a_{j_1,j_2} \right\|^2 \\
< & \frac{4}{H} \left\| \frac{1}{n_1 n_2}\sum^{n_1}_{j_1=1}\sum^{n_2}_{j_2=1} a_{j_1,j_2}^\ast a_{j_1,j_2} \right\| + \frac{8}{H}\Bigg\{\sum^{h_1}_{d_1=1} \left\|  \frac{1}{n_1 n_2}\sum^{n_1}_{j_1=1} \sum^{n_2}_{j_2=1} a_{j_1,j_2}^\ast a_{j_1+d_1,j_2} \right\|\\
+   &  \sum_{d_2=1}^{h_2} \left\| \frac{1}{n_1 n_2}\sum^{n_1}_{j_1=1}\sum_{j_2=1}^{n_2} a_{j_1,j_2}^\ast a_{j_1,j_2+d_2}\right\| + \sum^{h_1}_{d_1=1} \sum_{d_2=1}^{h_2} \left\| \frac{1}{n_1 n_2}\sum^{n_1}_{j_1=1}\sum_{j_2=1}^{n_2} a_{j_1,j_2}^\ast a_{j_1+d_1,j_2+d_2} \right\| \\
+   &  \sum^{h_1}_{d_1=1} \sum_{d_2=1}^{h_2} \left\| \frac{1}{n_1 n_2}\sum^{n_1}_{j_1=1}\sum_{j_2=1}^{n_2} a_{j_1+d_1,j_2}^\ast a_{j_1,j_2+d_2} \right\|\Bigg\}.
\end{split}\ee
\end{lem}

We need two basic formulas without proof, which are variants of Formula 8.2 and Formula 8.4 in \cite{NSZ2005}.

{\bf Formula 1.}  If $1 \le n \ge h \ge 0$ are natural numbers and $a_1, \cdots , a_n$
are elements of a $\ast$-algebra then, putting $a_j = 0$
for $j \le 0$ and for $j \ge n + 1$, we have
$$ (h + 1)\sum^n_{j=1} a_j = \sum^{n+h}_{k=1}\sum^{k}_{j=k-h} a_j.$$

{\bf Formula 2.}  If $1 \le n \ge h \ge 0$ are natural numbers and $a_{j,j'}$, $1 \le j, j' \le n$
are elements of a $\ast$-algebra, putting $a_{j,j'} = 0$ for $j$ or $j' \le 0$ and for $j$ or $j' \ge n + 1$, then we have
$$\sum^{n+h}_{k=1} \sum^k_{j,j'=k-h} a_{j,j'} = (h + 1)\sum^n_{j=1}a_{j,j}+
\sum^h_{d=1} (h - d + 1) \sum^n_{j=1} (a_{j,j+d} + a_{j+d,j}).$$

\begin{proof}
Using twice Formula 1, we have
\be\begin{split}
(h_1 +1)(h_2 +1) \sum_{j_1=1}^{n_1}\sum_{j_2=1}^{n_2} a_{j_1,j_2} & = (h_1 +1) \sum_{j_1 =1}^{n_1} \left( (h_2 +1) \sum_{j_2 =1}^{n_2} a_{j_1,j_2} \right) \\
& = (h_1 +1) \sum_{j_1 =1}^{n_1} \left(  \sum^{n_2+h_2}_{k_2=1}\sum^{k_2}_{j_2=k_2-h_2} a_{j_1,j_2} \right)  \\
& = \sum^{n_1+h_1}_{k_1=1}\sum^{n_2+h_2}_{k_2=1}\sum^{k_1}_{j_1=k_1-h_1}\sum^{k_2}_{j_2=k_2-h_2} a_{j_1,j_2} .
\end{split}\ee
Then by the Kadison-Schwarz inequality, we have
\be\begin{split}
  &(h_1 +1)^2 (h_2 +1)^2 \left( \sum_{j_1=1}^{n_1} \sum_{j_2=1}^{n_2} a_{j_1,j_2} \right)^\ast \left( \sum_{j_1=1}^{n_1}\sum_{j_2=1}^{n_2} a_{j_1,j_2} \right)\\
= & \left( \sum^{n_1+h_1}_{k_1=1}\sum^{n_2+h_2}_{k_2=1}\sum^{k_1}_{j_1=k_1-h_1}\sum^{k_2}_{j_2=k_2-h_2} a_{j_1,j_2} \right)^\ast \left( \sum^{n_1+h_1}_{k_1=1}\sum^{n_2+h_2}_{k_2=1}\sum^{k_1}_{j_1=k_1-h_1}\sum^{k_2}_{j_2=k_2-h_2} a_{j_1,j_2}  \right) \\
\le & (n_1+h_1)(n_2+h_2) \sum^{n_1+h_1}_{k_1=1}\sum^{n_2+h_2}_{k_2=1} \left| \sum^{k_1}_{j_1=k_1-h_1}\sum^{k_2}_{j_2=k_2-h_2} a_{j_1,j_2} \right|^2 \\
= & (n_1+h_1)(n_2+h_2) \sum^{n_1+h_1}_{k_1=1}\sum^{n_2+h_2}_{k_2=1} \sum^{k_1}_{j_1,j_1'=k_1-h_1}\sum^{k_2}_{j_2,j_2'=k_2-h_2} a_{j_1,j_2}^\ast a_{j_1',j_2'} \triangleq A.
\end{split}\ee
Denote $N_h=(n_1+h_1)(n_2+h_2)$ and use Formula 2, we get
\be\begin{split}
A= & N_h \sum^{n_1+h_1}_{k_1=1}\sum^{k_1}_{j_1,j_1'=k_1-h_1} \sum^{n_2+h_2}_{k_2=1}\sum^{k_2}_{j_2,j_2'=k_2-h_2} a_{j_1,j_2}^\ast a_{j_1',j_2'} \\
= & N_h(h_2 +1) \sum^{n_1+h_1}_{k_1=1}\sum^{k_1}_{j_1,j_1'=k_1-h_1} \sum_{j_2 =1}^{n_2}a_{j_1,j_2}^\ast a_{j_1',j_2} \\
+ & N_h \sum^{n_1+h_1}_{k_1=1}\sum^{k_1}_{j_1,j_1'=k_1-h_1} \sum_{d_2=1}^{h_2}(h_2-d_2 +1)\sum_{j_2=1}^{n_2}(a_{j_1,j_2}^\ast a_{j_1',j_2+d_2} + a_{j_1,j_2+d_2}^\ast a_{j_1',j_2})\\
\triangleq & N_h(h_2 +1) \sum_{j_2 =1}^{n_2} \mathrm{I}_{j_2} + N_h \sum_{d_2=1}^{h_2}(h_2-d_2 +1)\sum_{j_2=1}^{n_2} ( \mathrm{II}_{j_2,d_2} + \mathrm{II}^\ast_{j_2,d_2} ).
\end{split}\ee
Now use once more Formula 2, we have respectively
\be\begin{split} \mathrm{I}_{j_2} =& \sum^{n_1+h_1}_{k_1=1}\sum^{k_1}_{j_1,j_1'=k_1-h_1}a_{j_1,j_2}^\ast a_{j_1',j_2}\\
= & (h_1 +1)\sum^{n_1}_{j_1=1} a_{j_1,j_2}^\ast a_{j_1,j_2} +
  \sum^{h_1}_{d_1=1} (h_1 - d_1 + 1) \sum^{n_1}_{j_1=1} (a_{j_1,j_2}^\ast a_{j_1+d_1,j_2} + a_{j_1+d_1,j_2}^\ast a_{j_1,j_2}) ,
\end{split}\ee
and
\be\begin{split} \mathrm{II}_{j_2,d_2} = & \sum^{n_1+h_1}_{k_1=1}\sum^{k_1}_{j_1,j_1'=k_1-h_1} a_{j_1,j_2}^\ast a_{j_1',j_2+d_2}\\
= &  (h_1 + 1) \sum^{n_1}_{j_1=1} a_{j_1,j_2}^\ast a_{j_1,j_2+d_2} + \sum^{h_1}_{d_1=1} (h_1 - d_1 + 1) \sum^{n_1}_{j_1=1} ( a_{j_1,j_2}^\ast a_{j_1+d_1,j_2+d_2} + a_{j_1+d_1,j_2}^\ast a_{j_1,j_2+d_2} ) .
\end{split}\ee

Take the norm, use the assumption $n_i\geq h_i$ ($i=1,2$), and do some simple calculations,
\be\begin{split}
    & \left\| \frac{1}{n_1 n_2} \sum_{j_1=1}^{n_1} \sum_{j_2=1}^{n_2} a_{j_1,j_2} \right\|^2 \\
<   & \frac{N_h}{H n_1 n_2} \left\| \frac{1}{n_1 n_2}\sum^{n_1}_{j_1=1}\sum^{n_2}_{j_2=1} a_{j_1,j_2}^\ast a_{j_1,j_2} \right\| +
\frac{2N_h}{H n_1 n_2} \sum^{h_1}_{d_1=1} \left\|  \frac{1}{n_1 n_2}\sum^{n_1}_{j_1=1} \sum^{n_2}_{j_2=1} a_{j_1,j_2}^\ast a_{j_1+d_1,j_2} \right\| \\
+   & \frac{2N_h}{(h_1+1)H n_1 n_2} \sum_{d_2=1}^{h_2} \left\| \sum_{j_2=1}^{n_2}\frac{1}{n_1 n_2} \mathrm{II}_{j_2,d_2} \right\| \\
< & \frac{4}{H} \left\| \frac{1}{n_1 n_2}\sum^{n_1}_{j_1=1}\sum^{n_2}_{j_2=1} a_{j_1,j_2}^\ast a_{j_1,j_2} \right\| + \frac{8}{H}\sum^{h_1}_{d_1=1} \left\|  \frac{1}{n_1 n_2}\sum^{n_1}_{j_1=1} \sum^{n_2}_{j_2=1} a_{j_1,j_2}^\ast a_{j_1+d_1,j_2} \right\|\\
+   & \frac{8}{H} \sum_{d_2=1}^{h_2} \left\| \frac{1}{n_1 n_2}\sum^{n_1}_{j_1=1}\sum_{j_2=1}^{n_2} a_{j_1,j_2}^\ast a_{j_1,j_2+d_2}\right\| + \frac{8}{H} \sum^{h_1}_{d_1=1} \sum_{d_2=1}^{h_2} \left\| \frac{1}{n_1 n_2}\sum^{n_1}_{j_1=1}\sum_{j_2=1}^{n_2} a_{j_1,j_2}^\ast a_{j_1+d_1,j_2+d_2} \right\| \\
+   & \frac{8}{H} \sum^{h_1}_{d_1=1} \sum_{d_2=1}^{h_2} \left\| \frac{1}{n_1 n_2}\sum^{n_1}_{j_1=1}\sum_{j_2=1}^{n_2} a_{j_1+d_1,j_2}^\ast a_{j_1,j_2+d_2} \right\|.
\end{split}\ee

\end{proof}

With the above preparation, we can give the proof of our second main result now.
\begin{proof}[Proof of Theorem \ref{th:Main 2}]

Let $x \in \mathcal{K}^\bot \cap \M$ and fix $\e>0$. Use the multi-parameter ergodic theorem in \cite{JX2007} and ergodic property of $(T_1,T_2)$ (also c.f. Proposition 5.1 \cite{L2014}), we have $M_{n_1,n_2}(x) \rightarrow \tau(x)\cdot \I$ a.u. In other words, we can construct a projection $e\in P(\M)$ such that $\tau(e^\bot) \le \e$, and $$ \left\{M_{n_1, n_2}\left(T_1^{k_1}T_2^{k_2}(x^\ast) T_1^{s_1}T_2^{s_2}(x)\right)e \right\} \text{ converges to } \tau\left(T_1^{k_1}T_2^{k_2}(x^\ast) T_1^{s_1}T_2^{s_2}(x)\right)e = \widehat{\sigma_x}(k_1-s_1,k_2-s_2)e$$ in $\M$ for every $k_1,k_2,s_1,s_2 \in \mathbb{N}.$

Then, let $a_{j_1,j_2}= \lambda_1^{j_1}\lambda_2^{j_2}T_1^{j_1} T_2^{j_2}(x) e$ and employing Lemma \ref{le:VanderCorput}, we have
\be\begin{split}
    & \sup_{\lambda_1, \lambda_2 \in \mathbb{T}} \left\| M_{n_1,n_2}\big(x, (\lambda_1,\lambda_2)\big)e \right\|^2_\8 \\
\le & \frac{4}{H} \left\| e M_{n_1,n_2}(x^\ast x) e \right\|_\8 + \frac{8}{H}\sum^{h_1}_{d_1=1} \left\|  e M_{n_1,n_2}\big( x^\ast T_1^{d_1}(x) \big) e\right\|_\8\\
+   & \frac{8}{H} \sum_{d_2=1}^{h_2} \left\| e M_{n_1,n_2}\big( x^\ast T_2^{d_2}(x) \big) e\right\|_\8 + \frac{8}{H} \sum^{h_1}_{d_1=1} \sum_{d_2=1}^{h_2} \left\| e M_{n_1,n_2}\big( x^\ast T_1^{d_1}T_2^{d_2}(x) \big) e \right\|_\8 \\
+   & \frac{8}{H} \sum^{h_1}_{d_1=1} \sum_{d_2=1}^{h_2} \left\| e M_{n_1,n_2}\big( T_1^{d_1}(x^\ast)T_2^{d_2}(x) \big) e  \right\|_\8.
\end{split}\ee
Therefore, for fixed $h_1, h_2$, we have
\be\begin{split}
& \lim_{n_1,n_2} \sup_{\lambda_1, \lambda_2 \in \mathbb{T}} \left\| M_{n_1,n_2}\big(x, (\lambda_1,\lambda_2)\big) e\right\|^2_\8 \\
\le & \frac{4}{H} \|x\|_2^2 + \frac{8}{H}\sum^{h_1}_{d_1=1} |\widehat{\sigma_x}(-d_1,0)| + \frac{8}{H} \sum_{d_2=1}^{h_2}|\widehat{\sigma_x}(0,-d_2)| \\
+   & \frac{8}{H} \sum^{h_1}_{d_1=1} \sum_{d_2=1}^{h_2}|\widehat{\sigma_x}(-d_1,-d_2)| +  \frac{8}{H} \sum^{h_1}_{d_1=1} \sum_{d_2=1}^{h_2}|\widehat{\sigma_x}(d_1,-d_2)|\\
\le & \frac{4}{(h_1+1)(h_2 +1)} \|x\|_2^2 + \frac{8}{(h_1+1)(h_2 +1)} \sum^{h_1}_{l_1=-h_1} \sum_{l_2=-h_2}^{h_2}|\widehat{\sigma_x}(l_1,l_2)|.
\end{split}\ee

On the other hand, applying Wiener's criterion (c.f. Section 7.13\cite{Katz1976}) of continuity to measure $\sigma_x$, we have
\be
\lim_{h_1,h_2} \frac{1}{(h_1+1)(h_2 +1)} \sum^{h_1}_{l_1=-h_1} \sum_{l_2=-h_2}^{h_2}|\widehat{\sigma_x}(l_1,l_2)|^2=0,
\ee
thus
\be
\lim_{h_1,h_2} \frac{1}{(h_1+1)(h_2 +1)} \sum^{h_1}_{l_1=-h_1} \sum_{l_2=-h_2}^{h_2}|\widehat{\sigma_x}(l_1,l_2)|=0.
\ee
Thereupon, we conclude
\be
 \lim_{n_1,n_2} \sup_{\lambda_1, \lambda_2 \in \mathbb{T}} \left\| M_{n_1,n_2}\big(x, (\lambda_1,\lambda_2)\big)e \right\|^2_\8 = 0.
\ee
\end{proof}

\section{A characterization of multi-parameter bounded Besicovitch class}

In this section we first recall some preliminaries associated to the multi-parameter bounded Besicovitch class, then we give a characterization of this class as in the one-parameter case \cite{BL1985}, from which  we conclude that the class $\mathcal{D}$ is strictly larger than the bounded Besicovitch class.

\begin{definition}\label{def:Marcinkiewicz}
Let $A: \mathbb{Z}^d \rightarrow \mathbb{C} $ be a function. For $1 \le p < \8$, define
\be
\| A \|_p^p = \limsup_{\mathbf{n}} \frac{1}{ | \mathbf{n+1} |} \sum_{\mathbf{k} = - \mathbf{n}}^{\mathbf{n}} | A \left( \mathbf{k} \right) |^p.
\ee

We define the Marcinkiewicz space of order p, $$\mathfrak{M}_p = \{ A : ~ \| A \|_p < \8 \} $$ and the space $\mathfrak{M}_{\8} =\{A: \| A \|_{\8} = \sup_{\mathbf{k}} | A \left( \mathbf{k} \right) |< \8 \}$.
\end{definition}

\begin{rk} It is easy to see that $\mathfrak{M}_p$ is a vector space and $\| \cdot \|_p$ is a seminorm on it. A general definition of Marcinkiewicz space and more related results can be found in \cite{Ber1966}\cite{Khac1987}. We note that the H\"older inequality is valid for this seminorm. Specifically, we define a semi-inner product on $ \mathfrak{M}_{2}$, that is, for any $A,~B \in \mathfrak{M}_{2}$,
$$
\langle A , B \rangle = \lim_{\mathbf{n}} \frac{1}{ | \mathbf{n+1} |} \sum_{\mathbf{k} = - \mathbf{n}}^{\mathbf{n}} \overline{A ( \mathbf{k} )} B ( \mathbf{k} )
$$ if the limit exists.
We can also define a more general correlation on function $A$, as $\gamma_A(\mathbf{m})= \langle A, A^{\mathbf{m}}\rangle$ if it exists for every $\mathbf{m}\in \mathbb{Z}^d$, where $A^{\mathbf{m}}$ is $\mathbf{m}$-translation of $A$.
\end{rk}

Concerning translation and semi-inner product, we have the following simple results, which actually confirm some facts in Remark \ref{rk:spectralmeasure}:

\begin{prop}\label{prop:translation}
For any $A,~B \in \mathfrak{M}_{2}$, we have

\begin{enumerate}[{\rm (1)}]

\item $\| A^{\mathbf{m}} \|_2 = \| A \|_2 $, for any $\mathbf{m} \in \mathbb{Z}^d$;

\item $\langle A^{\mathbf{j}} , B^{\mathbf{m}} \rangle = \langle A , B^{\mathbf{m} - \mathbf{j}} \rangle = \langle A^{\mathbf{j} - \mathbf{m}} , B \rangle $, for any $\mathbf{m},~ \mathbf{j} \in \mathbb{Z}^d;$
\item If a function $A$ has a correlation as $\gamma_A $, then $\gamma_A $ is a positive definite function on $\mathbb{Z}^d$. In particular, $\gamma_A (-\mathbf{m}) = \overline{\gamma_A (\mathbf{m})}$ for every $\mathbf{m}\in \mathbb{Z}^d$.
\end{enumerate}
\end{prop}

\begin{proof}
(1) is a direct result by the definition, and (2) is easy to verify.

 We only prove (3). Taking any $z_1,z_2,\cdots,z_N \in \mathbb{C}$, any $\mathbf{m}_1,\mathbf{m}_2, \cdots, \mathbf{m}_N \in \mathbb{Z}^d,$ and any positive integer $N$, by (2) we have
\be\begin{split}
 \sum_{i=1}^N \sum_{j=1}^N \overline{z}_i z_j \gamma_A ( \mathbf{m}_j - \mathbf{m}_i )&=\sum_{i=1}^N \sum_{j=1}^N \overline{z}_i z_j \langle A^{\mathbf{m}_i} , A^{\mathbf{m}_j} \rangle\\
&= \langle \sum_{i=1}^N z_i A^{\mathbf{m}_i} , \sum_{j=1}^N z_j A^{\mathbf{m}_j} \rangle \geq0,
\end{split}\ee
which finishes the proof.
\end{proof}

Let $\mathcal{H}_0$ be the space of all {\it trigonometric polynomials} as $\mathbf{P}= \{P(\mathbf{k})\}_{\mathbf{k}\in \mathbb{N}^d},$ where $P(\mathbf{k})= \sum_{\alpha}c_{\alpha} \mathbf{z}_{\alpha}^{\mathbf{k}}$ (here $\mathbf{z}_{\alpha}\in \mathbb{T}^d$ and the sum is finite).
Let $\mathbf{b}=\{b(\mathbf{k})\}_{\mathbf{k}\in \mathbb{N}^d}$ be a multi-parameter complex sequence. We say that $$ \mathbf{b}\text{ {\it has a mean} if  } \lim_{\mathbf{n}}\frac{1}{|\mathbf{n+1}|}\sum_{\mathbf{k}=\mathbf{0}}^{\mathbf{n}}b(\mathbf{k})\text{  exists.}$$

Now we define the bounded Besicovitch class firstly introduced in \cite{JO1994}.
\begin{definition}\label{def:Besicovitch}
Let $1\le p < \8$. A function $\mathbf{a}:\mathbb{N}^d \rightarrow \mathbb{C}$ belongs to the {\it Besicovitch class} $B(p)$ (also called p-Besicovitch) if for each $\ep>0$ there exists a trigonometric polynomial $\mathbf{P}=\{P_{\ep}(\mathbf{k})\}$ such that
\beq\label{eq:Besicovitch}
\|\mathbf{a}-\mathbf{P}\|_p^p=\lim_{\mathbf{n}}\frac{1}{|\mathbf{n+1}|}\sum_{\mathbf{k=0}}^{\mathbf{n}}|a(\mathbf{k})-P_{\ep}(\mathbf{k})|^p < \ep.
\eeq
If $\mathbf{a}\in B(1)$, it is usually called Besicovitch function or Besicovitch sequence.

A function $\mathbf{a}$ is called a {\it p-bounded Besicovitch function} if $\mathbf{a}\in B(p)\cap \mathfrak{M}_{\8}$, and bounded Besicovitch if $\mathbf{a}\in B(1)\cap \mathfrak{M}_{\8}$.
\end{definition}

\begin{rk}\label{re:B(p) and mp}
It is proved by Jones and Olsen that for all $p\ge1,$ $B(p)\cap \mathfrak{M}_\8 = B(1)\cap\mathfrak{M}_\8$, and
the classes $\mathfrak{M}_p$ and $B(p)$ satisfy the following properties:
\begin{enumerate}[{\rm (1)}]

\item If $\{a(\mathbf{k})\} \in \mathfrak{M}_p$ and $\{b(\mathbf{k})\} \in \mathfrak{M}_q$, where $\frac{1}{p}+\frac{1}{q}=1$, and $c(\mathbf{k})=a(\mathbf{k})b(\mathbf{k})$, then $\{c(\mathbf{k})\}\in \mathfrak{M}_1.$

\item If $\{a_\el(\mathbf{k})\}$ has a mean for every $\el$, and $\{a_\el(\mathbf{k})\} \rightarrow \{a(\mathbf{k})\}$ in the semi-norm of space $\mathfrak{M}_1$, then $\{a(\mathbf{k})\}$ has a mean.

\item If $\{a_\el(\mathbf{k})\} \rightarrow \{a(\mathbf{k})\}$ in the $\mathfrak{M}_p$ semi-norm, and $\{a_\el(\mathbf{k})b(\mathbf{k})\}$ has a mean for every $\el$ where $\{b(\mathbf{k})\} \in \mathfrak{M}_q, ~ \frac{1}{p}+\frac{1}{q}=1$, then $\{a(\mathbf{k})b(\mathbf{k})\}$ has a mean.

\item For any sequence $\{a(\mathbf{k})\} \in B(p)$ and any $\mathbf{z}\in \mathbb{T}^d$, the sequence $\{a(\mathbf{k})\mathbf{z}^{\mathbf{k}}\}\in B(p)$ and has a mean. In particular, each sequence in $B(p)$ has a mean.
\end{enumerate}
\end{rk}

\begin{definition}\label{def:BFkernel}
Given a finite sequence of positive integers $\{n_1, n_2, \dots, n_r\}$ and a corresponding set of real numbers $\{\b_1, \b_2, \dots, \b_r\}$ which are linearly independent over the rationals, denote
$$ B= B(n_1,n_2,\dots, n_r; \b_1, \b_2, \dots, \b_r) $$
and denote by $K_B$ the kernel
\be
K_B(t)= \sum_{\nu_1=-n_1}^{n_1}\cdots \sum_{\nu_r=-n_r}^{n_r} \left( 1-\frac{|\nu_1|}{n_1} \right) \cdots \left( 1-\frac{|\nu_r|}{n_r} \right)
e^{2\pi i(\nu_1\b_1+ \cdots +\nu_r \b_r)t}.
\ee
This Kernel is called Bochner-Fej\'er Kernel, and for simplicity, we denote $$d_B(\nu_1, \dots, \nu_r)= \left( 1-\frac{|\nu_1|}{n_1} \right) \cdots \left( 1-\frac{|\nu_r|}{n_r} \right) .$$
In the $d$-parameter case, define $\mathbf{K_B(t)}= K_{B_1}(t_1)K_{B_2}(t_2)\cdots K_{B_d}(t_d)$, where each $B_j$ is as above and $\mathbf{t}=(t_1,\dots, t_d).$
Let $\mathbf{t=k}\in \mathbb{N}^d$, the Bochner-Fej\'er Kernel in discrete case becomes
\be
\mathbf{K_B(k)}= K_{B_1}(k_1)K_{B_2}(k_2)\cdots K_{B_d}(k_d) \in \mathcal{H}_0.
\ee
\end{definition}

\begin{rk}\label{rk:BFpolynomial}
If we define $\mathbf{K_B}\ast \mathbf{a}(\mathbf{k})= \lim_{\mathbf{N}} \frac{1}{\mathbf{N+1}} \sum_{\mathbf{k}=0}^{\mathbf{N}}\mathbf{K_B(k-j)}a (\mathbf{j})$, where $\mathbf{a}$ be any Besicovitch function, the following properties have been shown in \cite{JO1994}:
\begin{enumerate}[{\rm (1)}]
\item $\mathbf{K_B}\ast \mathbf{a}$ is a trigonometric polynomial;

\item $\| \mathbf{K_B}\ast \mathbf{a} \|_\8 \le \|\mathbf{a}\|_\8$;

\item $\| \mathbf{K_B}\ast \mathbf{a} \|_1 \le \|\mathbf{a}\|_1$;

\item  For any $\ep>0$, there is a Bochner-Fej\'er polynomial $ \mathbf{K_B}\ast \mathbf{a}$ such that $\| \mathbf{a}- \mathbf{K_B}\ast \mathbf{a}\|_\8 < \ep.$
\end{enumerate}
\end{rk}

With all the previous preparations, we give the following key result which characterizes the multi-parameter bounded Besicovitch sequences by spectral measure.

\begin{thm}\label{th:Besicovitch}
$\mathbf{a}= \{a(\mathbf{k})\}$ is bounded Besicovitch, i.e. $\mathbf{a}\in B(1)\cap \mathfrak{M}_\8$, if and only if $\mathbf{a}$ satisfies the following conditions:
\begin{enumerate}[{\rm (1)}]
\item $\mathbf{a}\in S\cap \mathfrak{M}_\8$ and the spectral measure $\sigma_{\mathbf{a}}$ is discrete;

\item the amplitude $\Gamma_\mathbf{a}(\mathbf{z})=\la \{z^{\mathbf{k}}\}, \mathbf{a} \ra= \lim_{\mathbf{n}}\frac{1}{\mathbf{n+1}}\sum_{\mathbf{k}=0}^{\mathbf{n}}a(\mathbf{k})\overline{\mathbf{z}}^{\mathbf{k}}$ exists for each $\mathbf{z}\in \mathbb{T}^d$;

\item $\sigma_\mathbf{a}(\{\mathbf{z}\})= |\Gamma_\mathbf{a}(\mathbf{z})|^2$ for all $\mathbf{z}\in \mathbb{T}^d$.
\end{enumerate}

\end{thm}

\begin{proof}
Assume first that $\mathbf{a}\in B(1)\cap \mathfrak{M}_\8$. From the definition of $B(1)$, we know there exists a sequence of trigonometric polynomials $\{\mathbf{P}_l\}$ that converges to $\mathbf{a}$ in the semi-norm $\|\cdot\|_1$, then by (2) of Remark \ref{re:B(p) and mp} we know $\mathbf{a}$ has a mean. Now consider $\mathbf{m}$-translation of $\mathbf{a}$ for any $\mathbf{m} \in \mathbb{N}^d.$ As $\mathbf{a}^{\mathbf{m}}\in B(1)$ so there also exists a sequence of trigonometric polynomials $\{\mathbf{P'}_s\}$ that converges to $\mathbf{a}^{\mathbf{m}}$ in the semi-norm $\|\cdot\|_1$. Thus using (3) of Remark \ref{re:B(p) and mp} twice we know $\gamma_\mathbf{a}(\mathbf{m})$ exists. So we conclude that $\mathbf{a}$ has a correlation. And the existence of $\Gamma_\mathbf{a}$ follows easily from (4) of Remark \ref{re:B(p) and mp}. It remains to check the second half of $(1)$ and $(3)$. We divide the proof into two steps.

Step 1. Let $\psi$ be a trigonometric polynomial $\psi =\{ \sum_{\a}c_\a \mathbf{z}_{\a}^{\mathbf{k}}\}$, then a direct calculation shows that its correlation $$\gamma_{\psi}(\mathbf{k})= \sum_{\a}|c_{\a}|^2 \mathbf{z}_{\a}^{\mathbf{k}}.$$
In particular, the spectral measure $\sigma_\psi$ is discrete and is given by $\sigma_\psi = \sum_{\a}|c_{\a}|^2 \delta_{\mathbf{z}_{\a}}$, and $\Gamma_\psi(\mathbf{z})= \sum_{\a}c_\a \delta_{\mathbf{z}_\a}(\mathbf{z})$ so $|\Gamma_\psi(\mathbf{z})|^2= \sum_{\a}|c_\a|^2 \delta_{\mathbf{z}_\a}(\mathbf{z})$. That is, $\sigma_\psi(\mathbf{z}) =|\Gamma_\psi(\mathbf{z})|^2$.

Step 2. By Remark \ref{rk:BFpolynomial}, let $\{\psi_n\}$ be a sequence of Bochner-Fej\'er polynomials such that
\be
\|\mathbf{a}-\psi_n\|_2\rightarrow 0 \text{ and } \|\psi_n\|_\8 \le \|\mathbf{a}\|_\8 \text{ for all }n\in \mathbb{N}.
\ee
It follows that $$\gamma_\mathbf{a}(\mathbf{m})= \la \mathbf{a}, \mathbf{a}^{\mathbf{m}} \ra = \lim_n \la \psi_n, \psi_n^{\mathbf{m}} \ra= \lim_n \gamma_{\psi_n} (\mathbf{m}) $$ uniformly in $\mathbf{m},$ and
$$ \Gamma_\mathbf{a}(\mathbf{z})= \la \{z^{\mathbf{k}}\}, \mathbf{a} \ra=\lim_n \la \{z^{\mathbf{k}}\},\psi_n \ra=\lim_n \Gamma_{\psi_n}(\mathbf{z}).$$
Thus $\sigma_\mathbf{a}(\{\mathbf{z}\})= \lim_n \sigma_{\psi_n}(\{\mathbf{z}\})=\lim_n |\Gamma_{\psi_n}(\mathbf{z})|^2=|\Gamma_\mathbf{a}(\mathbf{z})|^2$, so (3) is satisfied. Also we know $\gamma_\mathbf{a}$ is a uniformly almost periodic function (c.f. Theorem 8 P.3 \cite{Bes1954}), and the Bohr-Fourier coefficients are nonnegative. Thus by the Bochner-Fej\'er summing process as P.21 \cite{Bes1954}, the Bohr-Fourier series is absolutely and uniformly convergent, i.e., if $\gamma_\mathbf{a}(\mathbf{m})\sim \sum_{\a=1}^\8 C_\a \mathbf{z}_\a^{\mathbf{m}}$, then $\sum_{\a=1}^\8 |C_\a|=\sum_{\a=1}^\8 C_\a<\8$ and $\gamma_\mathbf{a}(\mathbf{m})= \sum_{\a=1}^\8 C_\a \mathbf{z}_\a^{\mathbf{m}}$ uniformly in $\mathbf{m}$. So $  \sigma_\mathbf{a}= \sum_{\a=1}^\8 C_\a \delta_{\mathbf{z}^\a}$ and (1) is totally satisfied for $\mathbf{a}.$

\

Assume now that $\mathbf{a}$ satisfies conditions (1), (2) and (3). By (1), there exists a sequence $\{\mathbf{z}_\a\}_{\a=1}^{\8}$ of distinct elements of $\mathbb{T}^d$, and a sequence of positive numbers $\{C_\a\}_{\a=1}^{\8}$ such that
$
\sigma_\mathbf{a}= \sum_{\a =1}^\8 C_\a \delta_{\mathbf{z}_\a}, ~ \ \sum_{\a =1}^\8 C_\a < +\8
$
and
$
\gamma_\mathbf{a}(\mathbf{m})= \sum_{\a =1}^\8 C_\a \mathbf{z}_\a^{\mathbf{m}}.
$
By conditions (2) and (3), $|\Gamma_\mathbf{a}(\mathbf{z})|^2= \sum_{\a =1}^\8 C_\a \delta_{\mathbf{z}_\a}(\mathbf{z}).$ Let $c_\a= {\Gamma_\mathbf{a}}(\mathbf{z}_\a)$, so $|c_\a|^2=C_\a$.

Define now $$a_N(\mathbf{k})= \sum_{\a = 1}^N c_\a \mathbf{z}_{\a}^{\mathbf{k}}.$$
To show that $\mathbf{a}\in B(2)$ it suffices to show that
\be
\|\mathbf{a}-\mathbf{a}_N\|_2^2 \rightarrow 0 ~~\text{ as } N \rightarrow \8,
\ee
which follows from
\be\begin{split}
\|\mathbf{a}-\mathbf{a}_N\|_2^2 & = \la a(\mathbf{k}) - \sum_{\a = 1}^N c_\a \mathbf{z}_{\a}^{\mathbf{k}},a(\mathbf{k}) - \sum_{\a = 1}^N c_\a \mathbf{z}_{\a}^{\mathbf{k}} \ra \\
& = \la \mathbf{a}, \mathbf{a}\ra -\sum_{\a = 1}^N c_\a \la \mathbf{a}, \{\mathbf{z}_{\a}^{\mathbf{k}}\} \ra - \sum_{\a = 1}^N \bar{c}_\a \la \{\mathbf{z}_{\a}^{\mathbf{k}}\}, \mathbf{a}\ra + \sum_{\a = 1}^N \sum_{\b = 1}^N\bar{c}_\a c_\b \la \{\mathbf{z}_{\a}^{\mathbf{k}}\},\{\mathbf{z}_{\b}^{\mathbf{k}}\} \ra\\
& = \gamma_\mathbf{a}(\mathbf{0}) - \sum_{\a = 1}^N c_\a \overline{\Gamma_\mathbf{a}(\mathbf{z}_\a)} - \sum_{\a = 1}^N \bar{c}_\a \Gamma_\mathbf{a}(\mathbf{z}_\a)+ \sum_{\a = 1}^N \sum_{\b = 1}^N\bar{c}_\a c_\b \delta_\a^\b\\
& = \sum_{\a =1}^\8 C_\a - \sum_{\a = 1}^N |c_\a|^2 - \sum_{\a = 1}^N |c_\a|^2 + \sum_{\a = 1}^N |c_\a|^2\\
& = \sum_{\a >N} C_\a \rightarrow 0~~\text{ as } N \rightarrow \8.
\end{split}\ee
This finishes the proof.

\end{proof}

\begin{example}\label{ex:Besi}
Let $a(\mathbf{k}) = (-1)^{[\log (k_1 + k_2 + \cdots + k_d +1)]}$(here $[x]$ denotes the largest integer $n\le x$). It is easy to see that the sequence $\mathbf{a}=\{a(\mathbf{k})\}$ has a correlation and that $\gamma_\mathbf{a}(\mathbf{m})=1$ for all $\mathbf{m} \in \mathbb{Z}^d.$ Thus the spectral measure $\sigma_\mathbf{a}$ corresponding to $\mathbf{a}$ is Dirac measure at $\mathbf{1} \in \mathbb{T}^d$, while $\sigma_\mathbf{a}$ is discrete. It is also not hard to check that the amplitude $\Gamma_\mathbf{a}(z)= \lim_{\mathbf{n}}\frac{1}{|\mathbf{n+1}|}\sum_{\mathbf{k}=0}^{\mathbf{n}}a(\mathbf{k})\bar{\mathbf{z}}^{\mathbf{k}}$ exists and equals $0$ for all $\mathbf{z}\in \mathbb{T}^d.$ Thus we know $\sigma_\mathbf{a}(\{\mathbf{1}\})=1$ while $|\Gamma_\mathbf{a}(1)|^2=0$. So the sequence $\mathbf{a}$ satisfies conditions (1) and (2) of Theorem \ref{th:Besicovitch}, but fails to satisfy condition (3).
\end{example}
Thus we have confirmed that the class of sequences $\mathcal{D}$ is strictly larger than the class of bounded Besicovitch sequences.

\
\noindent \textbf{Acknowledgement.} Guixiang Hong is supported in part by the ICMAT Severo Ochoa Grant SEV-2011-0087 (Spain).

\

\enlargethispage{2cm}

\vskip20pt

\hfill \noindent \textbf{Guixiang Hong} \\
\null \hfill School of Mathematics and Statistics \\
\null \hfill Wuhan University \\
\null \hfill Wuhan 430072. China \\
\null \hfill Instituto de Ciencias Matem{\'a}ticas \\ \null \hfill
CSIC-UAM-UC3M-UCM \\ \null \hfill Consejo Superior de
Investigaciones Cient{\'\i}ficas \\ \null \hfill C/ Nicol\'as Cabrera 13-15.
28049, Madrid. Spain \\ \null \hfill\texttt{guixiang.hong@whu.edu.cn}

\vskip2pt

\hfill \noindent \textbf{Mu Sun} \\
\null \hfill School of Mathematics and Statistics \\
\null \hfill Wuhan University \\
\null \hfill Wuhan 430072. China \\
\null \hfill\texttt{sunmu508@mails.ucas.ac.cn}

\end{document}